\documentclass[reqno, 12pt]{amsart}
\pdfoutput=1
\makeatletter
\let\origsection=\section \def\section{\@ifstar{\origsection*}{\mysection}}
\def\mysection{\@startsection{section}{1}\z@{.7\linespacing\@plus\linespacing}{.5\linespacing}{\normalfont\scshape\centering\S}}
\makeatother

\usepackage{amsmath,amssymb,amsthm}
\usepackage{mathrsfs}
\usepackage{mathabx}\changenotsign
\usepackage{dsfont}
\usepackage{bbm}

\usepackage{xcolor}
\usepackage[backref=section]{hyperref}
\usepackage[ocgcolorlinks]{ocgx2}
\hypersetup{
	colorlinks=true,
	linkcolor={red!60!black},
	citecolor={green!60!black},
	urlcolor={blue!60!black},
}

\usepackage[open,openlevel=2,atend]{bookmark}

\usepackage[abbrev,msc-links,backrefs]{amsrefs}
\usepackage{doi}

\renewcommand{\PrintDOI}[1]{\doi{#1}}

\usepackage[T1]{fontenc}
\usepackage{lmodern}
\usepackage[babel]{microtype}
\usepackage[english]{babel}

\usepackage{geometry}
\numberwithin{equation}{section}
\numberwithin{figure}{section}

\usepackage{enumitem}
\def\rmlabel{\upshape({\itshape \roman*\,})}

\def\alabel{\upshape({\itshape \alph*\,})}

\def\nlabel{\upshape({\itshape \arabic*\,})}

\let\polishlcross=\l
\def\l{\ifmmode\ell\else\polishlcross\fi}

\let\setminus=\smallsetminus

\makeatletter
\def\moverlay{\mathpalette\mov@rlay}
\def\mov@rlay#1#2{\leavevmode\vtop{   \baselineskip\z@skip \lineskiplimit-\maxdimen
		\ialign{\hfil$\m@th#1##$\hfil\cr#2\crcr}}}
\newcommand{\charfusion}[3][\mathord]{
	#1{\ifx#1\mathop\vphantom{#2}\fi
		\mathpalette\mov@rlay{#2\cr#3}
	}
	\ifx#1\mathop\expandafter\displaylimits\fi}
\makeatother

\newcommand{\dcup}{\charfusion[\mathbin]{\cup}{\cdot}}

\DeclareFontFamily{U}  {MnSymbolC}{}
\DeclareSymbolFont{MnSyC}         {U}  {MnSymbolC}{m}{n}
\DeclareFontShape{U}{MnSymbolC}{m}{n}{
	<-6>  MnSymbolC5
	<6-7>  MnSymbolC6
	<7-8>  MnSymbolC7
	<8-9>  MnSymbolC8
	<9-10> MnSymbolC9
	<10-12> MnSymbolC10
	<12->   MnSymbolC12}{}
\DeclareMathSymbol{\powerset}{\mathord}{MnSyC}{180}
\DeclareMathSymbol{\YY}{\mathord}{MnSyC}{42}

\usepackage{tikz}
\usetikzlibrary{calc,decorations.pathmorphing}
\usetikzlibrary{arrows,decorations.pathreplacing}
\pgfdeclarelayer{background}
\pgfdeclarelayer{foreground}
\pgfdeclarelayer{front}
\pgfsetlayers{background,main,foreground,front}

\usepackage{multicol}
\usepackage{subcaption}
\let\epsilon=\varepsilon
\let\eps=\epsilon
\let\phi=\varphi
\let\rho=\varrho
\let\theta=\vartheta

\def\NN{{\mathds N}}
\def\QQ{{\mathds Q}}

\def\ZZ{{\mathds Z}}

\def\llex{<_{\mathrm{lex}}}

\newcommand{\gS}{\mathfrak{S}}

\newtheoremstyle{note}  {4pt}  {4pt}  {\sl}  {}  {\bfseries}  {.}  {.5em}          {}
\newtheoremstyle{introthms}  {3pt}  {3pt}  {\itshape}  {}  {\bfseries}  {.}  {.5em}          {\thmnote{#3}}
\newtheoremstyle{remark}  {2pt}  {2pt}  {\rm}  {}  {\bfseries}  {.}  {.3em}          {}

\theoremstyle{plain}
\newtheorem{theorem}{Theorem}[section]
\newtheorem{lemma}[theorem]{Lemma}

\newtheorem{cor}[theorem]{Corollary}

\newtheorem{claim}[theorem]{Claim}

\theoremstyle{note}
\newtheorem{dfn}[theorem]{Definition}
\newtheorem{definition}[theorem]{Definition}

\theoremstyle{remark}
\newtheorem{remark}[theorem]{Remark}

\usepackage{lineno}
\newcommand*\patchAmsMathEnvironmentForLineno[1]{
	\expandafter\let\csname old#1\expandafter\endcsname\csname #1\endcsname
	\expandafter\let\csname oldend#1\expandafter\endcsname\csname end#1\endcsname
	\renewenvironment{#1}
	{\linenomath\csname old#1\endcsname}
	{\csname oldend#1\endcsname\endlinenomath}}
\newcommand*\patchBothAmsMathEnvironmentsForLineno[1]{
	\patchAmsMathEnvironmentForLineno{#1}
	\patchAmsMathEnvironmentForLineno{#1*}}
\AtBeginDocument{
	\patchBothAmsMathEnvironmentsForLineno{equation}
	\patchBothAmsMathEnvironmentsForLineno{align}
	\patchBothAmsMathEnvironmentsForLineno{flalign}
	\patchBothAmsMathEnvironmentsForLineno{alignat}
	\patchBothAmsMathEnvironmentsForLineno{gather}
	\patchBothAmsMathEnvironmentsForLineno{multline}
}

\usepackage{scalerel}

\makeatletter
\newcommand{\overrighharpoonup}[1]{\ThisStyle{%
		\vbox {\m@th\ialign{##\crcr
				\rightharpoonupfill \crcr
				\noalign{\kern-\p@\nointerlineskip}
				$\hfil\SavedStyle#1\hfil$\crcr}}}}

\def\rightharpoonupfill{%
	$\SavedStyle\m@th\mkern+0.8mu\cleaders\hbox{$\shortbar\mkern-4mu$}\hfill\rightharpoonuptip\mkern+0.8mu$}

\def\rightharpoonuptip{%
	\raisebox{\z@}[2pt][1pt]{\scalebox{0.55}{$\SavedStyle\rightharpoonup$}}}

\def\shortbar{%
	\smash{\scalebox{0.55}{$\SavedStyle\relbar$}}}
\makeatother

\makeatletter
\newcommand{\overlefharpoonup}[1]{\ThisStyle{%
		\vbox {\m@th\ialign{##\crcr
				\leftharpoonupfill \crcr
				\noalign{\kern-\p@\nointerlineskip}
				$\hfil\SavedStyle#1\hfil$\crcr}}}}

\def\leftharpoonupfill{%
	$\SavedStyle\m@th\mkern+0.8mu\cleaders\hbox{$\shortbar\mkern-4mu$}\hfill\leftharpoonuptip\mkern+0.8mu$}

\def\leftharpoonuptip{%
	\raisebox{\z@}[2pt][1pt]{\scalebox{0.55}{$\SavedStyle\leftharpoonup$}}}

\let\lra=\longrightarrow

\def\lefty{\mathrm{left}}
\def\righty{\mathrm{right}}
\def\LL{\mathrm{L}}
\def\Ra\mathrm{R}

\makeatletter
\newsavebox\myboxA
\newsavebox\myboxB
\newlength\mylenA

\newcommand*\xoverline[2][0.75]{%
	\sbox{\myboxA}{$\m@th#2$}%
	\setbox\myboxB\null
	\ht\myboxB=\ht\myboxA%
	\dp\myboxB=\dp\myboxA%
	\wd\myboxB=#1\wd\myboxA
	\sbox\myboxB{$\m@th\overline{\copy\myboxB}$}
	\setlength\mylenA{\the\wd\myboxA}
	\addtolength\mylenA{-\the\wd\myboxB}%
	\ifdim\wd\myboxB<\wd\myboxA%
	\rlap{\hskip 0.5\mylenA\usebox\myboxB}{\usebox\myboxA}%
	\else
	\hskip -0.5\mylenA\rlap{\usebox\myboxA}{\hskip 0.5\mylenA\usebox\myboxB}%
	\fi}
\makeatother

\DeclareSymbolFont{symbolsC}{U}{txsyc}{m}{n}
\SetSymbolFont{symbolsC}{bold}{U}{txsyc}{bx}{n}
\DeclareFontSubstitution{U}{txsyc}{m}{n}
\DeclareMathSymbol{\strictif}{\mathrel}{symbolsC}{74}

\DeclareSymbolFont{stmry}{U}{stmry}{m}{n}
\DeclareMathSymbol\arrownot\mathrel{stmry}{"58}
\DeclareMathSymbol\Arrownot\mathrel{stmry}{"59}
\def\longarrownot{\mathrel{\mkern5.5mu\arrownot\mkern-5.5mu}}

\def\nlra{\longarrownot\longrightarrow}

\DeclareMathOperator{\ER}{ER}

\begin{document}
\dedicatory{Dedicated to the memory of Ronald Graham}
\title{On quantitative aspects of a canonisation theorem for edge-orderings}

\author[Chr.~Reiher]{Christian Reiher}
\address{Fachbereich Mathematik, Universit\"at Hamburg, Hamburg, Germany}
\email{Christian.Reiher@uni-hamburg.de}
\email{schacht@math.uni-hamburg.de}
\email{kevinsames@gmx.de} 

\author[V.~R\"{o}dl]{Vojt\v{e}ch R\"{o}dl}
\address{Department of Mathematics, 
Emory University, Atlanta, USA}
\email{vrodl@emory.edu}
\email{marcelo.tadeu.sales@emory.edu}

\thanks{The second author is supported by NSF grant DMS 1764385.}

\author[M.~Sales]{Marcelo Sales}
\thanks{The third author is partially supported by NSF grant DMS 1764385.}

\author[K.~Sames]{Kevin Sames}

\author[M.~Schacht]{Mathias Schacht}
\thanks{The fifth author is supported by the ERC (PEPCo 724903)}

\begin{abstract}
	For integers $k\ge 2$ and $N\ge 2k+1$ there are $k!2^k$ {\it canonical orderings} of 
	the edges of the complete $k$-uniform hypergraph with vertex set $[N] = \{1,2,\dots, N\}$.
	These are exactly the orderings with the property that any two subsets $A, B\subseteq [N]$ 
	of the same size induce isomorphic suborderings. We study the associated {\it canonisation} problem 
	to estimate, given $k$ and $n$, the least integer $N$ such that no matter how the $k$-subsets 
	of $[N]$ are ordered there always exists an $n$-element set $X\subseteq [N]$ whose $k$-subsets 
	are ordered canonically. For fixed $k$ we prove lower and upper bounds on these numbers 
	that are $k$ times iterated exponential 
	in a polynomial of $n$. 
\end{abstract}

\maketitle

\section{Introduction}\label{sec:intro}

We consider numerical aspects of an unpublished result of Klaus Leeb from the early seventies
(see~\cites{NPRV85, NR17}). For a natural number $N$ we set $[N]=\{1,2,\dots,N\}$. 
Given a set $X$ and a nonnegative integer $k$ we write $X^{(k)}=\{e\subseteq X\colon |e|=k\}$ 
for the set of all $k$-element subsets of $X$.  

\subsection{Ramsey theory}
Recall that for any integers $k, n, r\ge 1$, Ramsey's theorem~\cite{Ramsey30} informs us 
that every sufficiently large integer $N$ satisfies the partition relation 
\begin{equation}\label{eq:1500}
	N\longrightarrow (n)^k_r\,,
\end{equation}
meaning that for every colouring of $[N]^{(k)}$ with $r$ colours there exists a set $X\subseteq [N]$
of size~$n$ such that~$X^{(k)}$ is monochromatic. The least number~$N$ validating~\eqref{eq:1500}
is denoted by~$\mathrm{R}^{(k)}(n, r)$. The {\it negative} partition relation $N\nlra (n)^k_r$
expresses the fact that~\eqref{eq:1500} fails.

For colourings $f\colon [N]^{(k)}\longrightarrow \NN$ with infinitely many colours, however, 
one can, in general, not even hope to obtain a monochromatic set of size $k+1$. 
For instance, it may happen 
that every $k$-element subset of $[N]$ has its own private colour. 
Nevertheless, Erd\H{o}s and Rado~\cites{ER52, ER50}
established a meaningful generalisation of Ramsey's theorem to colourings with infinitely many 
colours, even if the ground sets stay finite. In this context it is convenient to regard such 
colourings as equivalence relations, two $k$-sets being in the same equivalence class if and only 
if they are of the same colour. 
Now Erd\H{o}s and Rado~\cites{ER52, ER50} call an equivalence relation 
on $[N]^{(k)}$ {\it canonical} if for any two subsets $A, B\subseteq [N]$ of the same size 
the equivalence relations induced on~$A^{(k)}$ and~$B^{(k)}$ correspond to each other via 
the order preserving map between $A$ and $B$. So, roughly speaking, an equivalence relation is 
canonical if you cannot change its essential properties by restricting your attention to a subset. 
It turns out that for~$N\ge 2k+1$ there are always exactly $2^k$ canonical equivalence relations 
on $[N]^{(k)}$ that can be parametrised by subsets~$I\subseteq [k]$ in the following natural way. 
Let $x, y\in [N]^{(k)}$ be two $k$-sets and let $x=\{x_1, \dots, x_k\}$ 
as well as $y=\{y_1, \dots, y_k\}$ enumerate their elements in increasing order. 
We write $x\equiv_I y$ if and only if $x_i=y_i$ holds for all~$i\in I$.
Now $\{\equiv_I\colon I\subseteq [k]\}$ is the collection of all canonical equivalence 
relations on~$[N]^{(k)}$. The {\it canonical Ramsey theorem} of Erd\H{o}s and Rado 
asserts that given two integers~$k$ and~$n$ there exists an integer $\ER^{(k)}(n)$ such that 
for every $N\ge\ER^{(k)}(n)$ and every equivalence relation~$\equiv$ on~$[N]^{(k)}$ there exists a 
set $X\subseteq [N]$ of size $n$ with the property that $\equiv$ is canonical on~$X^{(k)}$. 

This result sparked the development of {\it canonical Ramsey theory} in the seventies. 
Let us commemorate some contributions to the area due to {\sc Ronald Graham}, 
who was among the main protagonists of Ramsey theory in the $20^{\mathrm{th}}$ century: Together 
with Erd\H{o}s~\cite{EG80}, he established a canonical version of van der Waerden's theorem, 
stating that if for $N\ge \mathrm{EG}(k)$ we colour $[N]$ with arbitrarily many colours, then there  
exists an arithmetic progression of length $k$ which is either monochromatic or receives~$k$ 
distinct colours. With Deuber, Pr\"omel, 
and Voigt~\cite{DGPV} he later obtained a canonical version of the Gallai-Witt theorem, 
which is more difficult to state --- in $d$ dimensions the canonical colourings 
for a given configuration $F\subseteq \ZZ^d$ are now parametrised by vector 
subspaces $U\subseteq \QQ^d$ having a basis of vectors of the form $x-y$ with $x, y\in F$.
Interestingly, {\sc Ronald Graham} was also among the small number of co-authors of Klaus 
Leeb~\cites{GLR, GLK2} and in joint work with Rothschild they settled a famous conjecture of Rota 
on Ramsey properties of finite vector spaces.

The canonisation discussed here concerns linear orderings. We say that a linear 
order $([N]^{(k)}, \strictif)$ is {\it canonical} if for any two sets $A, B\subseteq [N]$
of the same size the order preserving map from $A$ to $B$ induces an order preserving 
map from $(A^{(k)}, \strictif)$ to $(B^{(k)}, \strictif)$. It turns out that for $N\ge 2k+1$ 
there are exactly $k!2^k$ canonical orderings of $[N]^{(k)}$ that can uniformly be parametrised by 
pairs $(\eps, \sigma)$ consisting of a {\it sign vector} $\eps\in \{-1,+1\}^k$ and a 
permutation $\sigma\in \gS_k$, i.e., a bijective map $\sigma\colon [k]\lra [k]$. 

\begin{definition}\label{def:canonical}
	Let $N, k\ge 1$ be integers, and let $(\eps, \sigma)\in \{+1, -1\}^k\times \gS_k$ be a pair 
	consisting of a sign vector $\eps=(\eps_1, \dots, \eps_k)$ and a permutation $\sigma$ of $[k]$. 
	The ordering $\strictif$ on~$[N]^{(k)}$ 
	{\it associated with $(\eps, \sigma)$} is the unique ordering with the property that  
	if 
	\[
		1\le a_1< \dots <a_k\le N
		\quad \text{ and } \quad 1\le b_1< \dots <b_k\le N\,,
	\]
	then
	\[
		 \{a_1,\dots,a_k\} \strictif \{b_1,\dots, b_k\}
		 \,\,\, \Longleftrightarrow \,\,\,
		 (\epsilon_{\sigma(1)}a_{\sigma(1)},\dots,\epsilon_{\sigma(k)}a_{\sigma(k)})
		 \llex 
		 (\epsilon_{\sigma(1)}b_{\sigma(1)},\dots,\epsilon_{\sigma(k)}b_{\sigma(k)})\,,
	\]
	where $\llex$ denotes the lexicographic order on $\ZZ^k$.
\end{definition}

It is readily seen that such an ordering $\strictif$ associated with $(\eps, \sigma)$ does 
indeed exist, that it is uniquely determined, and, moreover, canonical. 
Conversely, for $N\ge 2k+1$ every canonical ordering is 
of this form (see Lemma~\ref{lem:canonical} below). 
We proceed by stating the aforementioned unpublished theorem due to Klaus Leeb.

\begin{theorem}[Leeb]\label{th:leeb}
	For all positive integers $k$ and $n$ there exists an integer $\LL^{(k)}(n)$ such that for 
	every~$N\ge \LL^{(k)}(n)$ and every ordering $([N]^{(k)}, \strictif)$
	there is a set $X\subseteq [N]$ of size $n$ such that $(X^{(k)}, \strictif)$ is canonical. 
\end{theorem}

In the sequel, $\LL^{(k)}(n)$ will always refer to the least number with this property. As we shall
prove, for fixed $k\ge 2$ the dependence of $\LL^{(k)}(n)$ on $n$ is $k$ times exponential in a 
polynomial of $n$.

\subsection{Bounds}
The study of the Ramsey numbers $\mathrm{R}^{(k)}(n, r)$ has received extensive 
interest (see e.g.~\cites{EHR65, CFS10, GRS90, EH72}). To state some of the currently known bounds, 
one recursively defines {\it tower functions} $t_k\colon\NN\lra\NN$ by 
\[
	t_0(x)=x
	\quad \text{ and } \quad 
	t_{i+1}(x)=2^{t_i(x)}\,.
\]
Arguing probabilistically Erd\H{o}s~\cite{Er47} showed 
that $\mathrm{R}^{(2)}(n, 2)\ge 2^{(n+1)/2}$ and Erd\H{o}s, Hajnal, and Rado \cite{EHR65} proved later 
that $\mathrm{R}^{(k)}(n, 2)\ge t_{k-2}(c_k n^2)$ for $k>2$, where $c_k$ is a positive constant. 
They also proved $\mathrm{R}^{(k)}(n, 4)\ge t_{k-1}(c_k n)$ for all $k\ge 2$, 
where $c_k>0$ is an absolute constant. In the other direction, it is known 
that $\mathrm{R}^{(k)}(n, r)\le t_{k-1}(C_{k, r}n)$ (see~\cite{ER52}).

Estimates on the canonical Ramsey numbers $\ER^{(k)}(n)$ were studied in \cites{LR95, Sh96}. 
Notably, Lefmann and R\"{o}dl proved the lower bound $\ER^{(k)}(n)\ge t_{k-1}(c_kn^2)$,
while Shelah obtained the complementary upper bound $\ER^{(k)}(n)\le t_{k-1}(C_kn^{8(2k-1)})$.

Let us now turn our attention to the Leeb numbers $\LL^{(k)}(n)$. For $k=1$, there exist
only two canonical orderings of $[N]^{(1)}$, namely the ``increasing'' and the ``decreasing''
one corresponding to the sign vectors $\eps=(+1)$ and $\eps=(-1)$, respectively. 
Thus the well-known Erd\H{o}s-Szekeres theorem~\cite{ES35} yields the exact 
value $\LL^{(1)}(n)=(n-1)^2+1$. Accordingly, Leeb's theorem 
can be viewed as a multidimensional version of the Erd\H{o}s-Szekeres theorem and we refer 
to~\cite{Lang} for further variations on this theme. Our own results can be summarised as follows.

\begin{theorem}\label{thm:lower}
	If $n\ge 4$ and $R\nlra (n-1)^2_2$, then $\LL^{(2)}(n)> 2^{R}$. Moreover, if $n\ge k\ge 3$, 
	then $\LL^{(k)}(4n+k) > 2^{\LL^{(k-1)}(n)-1}$. 
\end{theorem}

Due to the known lower bounds on diagonal Ramsey numbers this implies 
\[
	\LL^{(2)}(n)\ge t_2(n/2)
	\quad \text {as well as } \quad 
	\LL^{(k)}(n)\ge t_{k}(c_kn) \text{ for } k\ge 3\,. 
\]
We offer an upper bound with the same number of exponentiations.

\begin{theorem}\label{thm:upper}
	For every $k\ge 2$ there exists a constant $C_k$ such that $\LL^{(k)}(n)\le t_{k}(n^{C_k})$
	holds for every positive integer $n$.
\end{theorem}

The case $k=2$ of Theorems~\ref{thm:lower} and~\ref{thm:upper} was obtained independently by Conlon, 
Fox, and Sudakov in unpublished work~\cite{Fox}.

\subsection*{Organisation} We prove Theorem~\ref{thm:lower} in Section~\ref{sec:comb}.
The upper bound, Theorem~\ref{thm:upper}, is established in Section~\ref{sec:upper}. 
Lemma~\ref{lem:2150} from this proof turns out to be applicable to Erd\H{o}s-Rado 
numbers as well and Section~\ref{sec:ERS} illustrates this by showing a variant of 
Shelah's result, namely $\ER^{(k)}(n)\le t_{k-1}(C_kn^{6k})$. 

\section{Lower bounds: Constructions}\label{sec:comb}

\subsection{Trees and combs}
Our lower bound constructions take inspiration from the negative stepping-up lemma due
to Erd\H{o}s, Hajnal, and Rado \cite{EHR65}. Notice that in order to prove an inequality 
of the form $\LL^{(k)}(n)>N$ one needs to exhibit a special ordering $([N]^{(k)}, \strictif)$
for which there is no set $Z\subseteq [N]$ of size $n$ such that $\strictif$ orders $Z^{(k)}$ 
canonically. For us, $N=2^m$ will always be a power of two and for reasons of visualisability 
we identify the numbers in $[N]$ with the leaves of a binary tree of height $m$. 
We label the levels of our tree from the bottom to the top with the numbers in $[m]$. 
So the root is at level $1$ and the leaves are immediately above 
level $m$ (see Figure~\ref{fig:optimalline}).

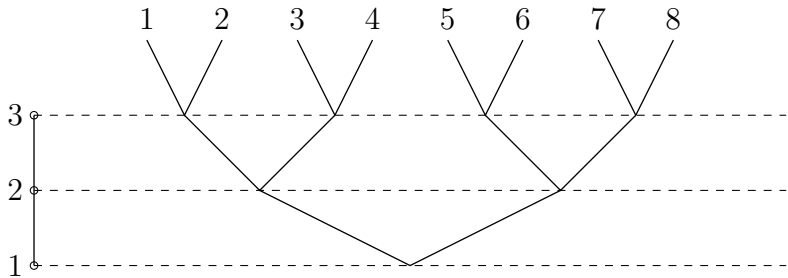
\begin{figure}[h]
\centering
{\hfil \begin{tikzpicture}[scale=1.0]
    \coordinate (A) at (0,0);
    \coordinate (B) at (-2,1);
    \coordinate (C) at (2,1);
    \coordinate (D) at (-3,2);
    \coordinate (E) at (-1,2);
	\coordinate (F) at (1,2);
    \coordinate (G) at (3,2);
    \coordinate [label=above:$1$] (H1) at (-3.5,3);
    \coordinate [label=above:$2$] (H2) at (-2.5,3);
    \coordinate [label=above:$3$] (H3) at (-1.5,3);
    \coordinate [label=above:$4$] (H4) at (-0.5,3);
    \coordinate [label=above:$5$] (H5) at (0.5,3);
    \coordinate [label=above:$6$] (H6) at (1.5,3);
    \coordinate [label=above:$7$] (H7) at (2.5,3);
    \coordinate [label=above:$8$] (H8) at (3.5,3);
    
    \coordinate [label=left:$1$] (L1) at (-5,0);
    \coordinate [label=left:$2$] (L2) at (-5,1);
    \coordinate [label=left:$3$] (L3) at (-5,2);
    \coordinate (R1) at (5,0);
    \coordinate (R2) at (5,1);
    \coordinate (R3) at (5,2);

    \draw[line width=0.5] (C)--(A)--(B);
    \draw[line width=0.5] (D)--(B)--(E);
    \draw[line width=0.5] (F)--(C)--(G);
    \draw[line width=0.5] (H1)--(D)--(H2);
    \draw[line width=0.5] (H3)--(E)--(H4);
    \draw[line width=0.5] (H5)--(F)--(H6);
    \draw[line width=0.5] (H7)--(G)--(H8);
    \draw[line width=0.5] (L1)--(L2)--(L3);
    \draw[dashed] (L1)--(R1);
    \draw[dashed] (L2)--(R2);
    \draw[dashed] (L3)--(R3);
    \draw (L1) circle [radius=0.05];
 	\draw (L2) circle [radius=0.05];
 	\draw (L3) circle [radius=0.05];
    
    \end{tikzpicture}\hfil}
\caption{The binary tree and its levels for $m=3$.}
\label{fig:optimalline}
\end{figure}

Alternatively, we can identify $[N]$ with the set $\{+1, -1\}^m$ of all $\pm1$-vectors of length~$m$
such that the standard ordering on $[N]$ corresponds to the lexicographic one 
on~$\{+1, -1\}^m$. The literature often works with the number $0$ instead of $-1$ here, but for the 
application we have in mind our choice turns out to be advantageous.  

Now every set $X\subseteq [N]$ generates a tree $T_X$, which is the subtree of our binary tree 
given by the union of the paths from the leaves in $X$ to the root (see Figure~\ref{fig:2057}).

\begin{figure}[h]
\centering
{\hfil \begin{tikzpicture}[scale=1.0]
    \coordinate (A) at (0,0);
    \coordinate (B) at (-2,1);
    \coordinate (C) at (2,1);
    \coordinate (D) at (-3,2);
    \coordinate (E) at (-1,2);
	\coordinate (F) at (1,2);
    \coordinate (G) at (3,2);
    \coordinate [label=above:$1$] (H1) at (-3.5,3);
    \coordinate [label=above:$2$] (H2) at (-2.5,3);
    \coordinate [label=above:$3$] (H3) at (-1.5,3);
    \coordinate [label=above:$4$] (H4) at (-0.5,3);
    \coordinate [label=above:$5$] (H5) at (0.5,3);
    \coordinate [label=above:$6$] (H6) at (1.5,3);
    \coordinate [label=above:$7$] (H7) at (2.5,3);
    \coordinate [label=above:$8$] (H8) at (3.5,3);
    
    \coordinate [label=left:$1$] (L1) at (-5,0);
    \coordinate [label=left:$2$] (L2) at (-5,1);
    \coordinate [label=left:$3$] (L3) at (-5,2);
    \coordinate (R1) at (5,0);
    \coordinate (R2) at (5,1);
    \coordinate (R3) at (5,2);

    \draw[line width=2.5] (C)--(A)--(B);
    \draw[line width=2.5] (D)--(B)--(E);
    \draw[line width=0.5] (F)--(C);
    \draw[line width=2.5] (G)--(C);
    \draw[line width=0.5] (H1)--(D);
    \draw[line width=2.5] (H2)--(D);
    \draw[line width=2.5] (H3)--(E);
    \draw[line width=0.5] (H4)--(E);
    \draw[line width=0.5] (H5)--(F)--(H6);
    \draw[line width=2.5] (H7)--(G);
    \draw[line width=0.5] (H8)--(G);
    \draw[line width=0.5] (L1)--(L2)--(L3);
    \draw[dashed] (L1)--(R1);
    \draw[dashed] (L2)--(R2);
    \draw[dashed] (L3)--(R3);
    \draw (L1) circle [radius=0.05];
 	\draw (L2) circle [radius=0.05];
 	\draw (L3) circle [radius=0.05];
 	\draw[fill] (H2) circle [radius=0.1];
 	\draw[fill] (H3) circle [radius=0.1];
 	\draw[fill] (H7) circle [radius=0.1];
 	\draw[fill] (B) circle [radius=0.1];
 	\draw[fill] (A) circle [radius=0.1];
    
    \end{tikzpicture}\hfil}
\caption{The auxiliary tree $T_X$ for $X=\{2,3,7\}$.}
\label{fig:2057}
\end{figure}
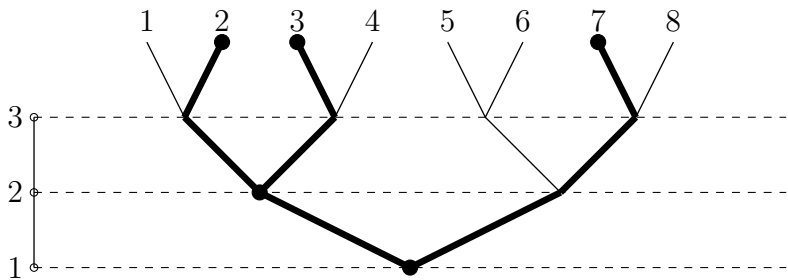

An essential observation on which both the double-exponential lower bound construction 
for the four-colour Ramsey number of triples and our lower bound on $\LL^{(2)}(n)$ rely 
is that there are two essentially different kinds of triples: those engendering a ``left tree'' 
and those engendering a ``right tree''. Roughly speaking the difference is that when descending
from the leaves to the root, the two elements concurring first are the two smaller 
ones for left trees and the two larger ones for right trees. For instance, Figure~\ref{fig:2057} 
displays a left tree.

To make these notions more precise we introduce the following notation. Given at least two distinct 
vectors $x_1, \dots, x_t\in \{-1, +1\}^m$ with coordinates $x_i=(x_{i1}, \dots, x_{im})$ we write 
\[
	\delta(x_1, \dots, x_t)=\min\{\mu\in [m]\colon x_{1\mu}= \dots = x_{t\mu} \text{ is not the case}\}\,.
\]
for the first level where two of them differ. 

Now let $xyz\in [N]^{(3)}$ be an arbitrary triple with $x<y<z$ and set $\delta=\delta(x, y, z)$. 
This means that $x$, $y$, and $z$ agree in the first $\delta-1$ coordinates, while their entries 
in the $\delta^{\mathrm{th}}$ coordinate, which we denote by $x_\delta$, $y_\delta$, $z_\delta$,
fail to coincide. As a consequence of $x<y<z$ we know $x_\delta\le y_\delta\le z_\delta$. Thus 
$x_\delta=-1$, $z_\delta=+1$, and there remain two possibilities: Depending on 
whether $y_\delta=-1$ or $y_\delta=+1$ we say that $xyz$ forms a {\it left tree} or 
a {\it right tree}, respectively. Equivalently, $xyz$ yields a left tree 
whenever $\delta(x, y)>\delta(y, z)$ and a right tree whenever $\delta(x, y)<\delta(y, z)$.  

When the triples are coloured with $\{\lefty, \righty\}$ depending on whether they form left
trees or right trees, the monochromatic sets are called {\it combs}. 
More explicitly, for 
\[	
	1\le x_1<\dots < x_t\le N
\]
the set $X=\{x_1, \dots, x_t\}$ is a {\it left comb} if 
$\delta(x_1, x_2)>\delta(x_2, x_3)>\dots >\delta(x_{t-1}, x_t)$
and a {\it right comb} if 
$\delta(x_1, x_2)<\delta(x_2, x_3)<\dots <\delta(x_{t-1}, x_t)$ (see Figure~\ref{fig:0001}). 
For instance, the empty set, every singleton, and every pair are both a left comb and a 
right comb. Since every 
triple is either a left tree or a right tree, triples are combs in exactly one of the two 
directions. Some quadruples fail to be combs.   

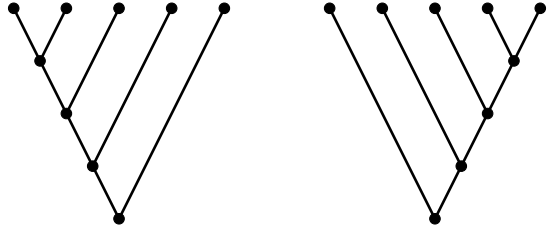
\begin{figure}[h]
\centering
{\hfil \begin{tikzpicture}[scale=0.7]
    \coordinate (R1) at (-3,0);
    \coordinate (R2) at (-3.5,1);
    \coordinate (R3) at (-4,2);
    \coordinate (R4) at (-4.5,3);
    \coordinate (R5) at (-5,4);
	\coordinate (R6) at (-4,4);
    \coordinate (R7) at (-3,4);
    \coordinate (R8) at (-2,4);
    \coordinate (R9) at (-1,4);
    \coordinate (H1) at (3,0);
    \coordinate (H2) at (3.5,1);
    \coordinate (H3) at (4,2);
    \coordinate (H4) at (4.5,3);
    \coordinate (H5) at (5,4);
    \coordinate (H6) at (4,4);
    \coordinate (H7) at (3,4);
    \coordinate (H8) at (2,4);
    \coordinate (H9) at (1,4);

    \draw[line width=1] (R1)--(R2)--(R3)--(R4)--(R5);
    \draw[line width=1] (R1)--(R9);
    \draw[line width=1] (R2)--(R8);
    \draw[line width=1] (R3)--(R7);
    \draw[line width=1] (R4)--(R6);
    
    \draw[line width=1] (H1)--(H2)--(H3)--(H4)--(H5);
    \draw[line width=1] (H1)--(H9);
    \draw[line width=1] (H2)--(H8);
    \draw[line width=1] (H3)--(H7);
    \draw[line width=1] (H4)--(H6);
    
    \draw[fill] (R1) circle [radius=0.1];
 	\draw[fill] (R2) circle [radius=0.1];
 	\draw[fill] (R3) circle [radius=0.1];
    \draw[fill] (R4) circle [radius=0.1];
 	\draw[fill] (R5) circle [radius=0.1];
 	\draw[fill] (R6) circle [radius=0.1];
    \draw[fill] (R7) circle [radius=0.1];
 	\draw[fill] (R8) circle [radius=0.1];
 	\draw[fill] (R9) circle [radius=0.1];
    \draw[fill] (H1) circle [radius=0.1];
 	\draw[fill] (H2) circle [radius=0.1];
 	\draw[fill] (H3) circle [radius=0.1];
    \draw[fill] (H4) circle [radius=0.1];
 	\draw[fill] (H5) circle [radius=0.1];
 	\draw[fill] (H6) circle [radius=0.1];
 	\draw[fill] (H7) circle [radius=0.1];
 	\draw[fill] (H8) circle [radius=0.1];
 	\draw[fill] (H9) circle [radius=0.1];

    \end{tikzpicture}\hfil}
\caption{A left comb and a right comb.}
\label{fig:0001}
\end{figure}

\subsection{Ordering pairs}
Let us recall that for Ramsey numbers the negative stepping-up lemma shows, e.g., that if 
$m\nlra (n-1)^2_2$, then $2^m\nlra (n)^3_4$. As we shall reuse the argument, we provide a brief
sketch. Let $f\colon [2^m]^{(3)}\lra \{\lefty, \righty\}$ be the colouring telling
us whether a triple generates a left tree or a right tree and 
let $g\colon [m]^{(2)}\lra \{+1, -1\}$ be a colouring exemplifying the negative partition 
relation $m\nlra (n-1)^2_2$. Define a 
colouring 
\[
	h\colon [2^m]^{(3)}\lra \{\lefty, \righty\}\times \{+1, -1\}
\]
of the triples with four colours by 
\[
	h(xyz)=\bigl(f(xyz), g\bigl(\delta(xy), \delta(yz)\bigr)\bigr)
\]
whenever $1\le x<y<z\le 2^m$ (where, for transparency, we abbreviated $\delta(x, y)$ 
to~$\delta(xy)$).
If $Z\subseteq [2^m]$ is monochromatic with respect to $h$, then, in particular, it needs to be 
monochromatic for $f$, i.e., $Z$ is a comb. Thus, if $Z=\{z_1, \dots, z_{|Z|}\}$ 
with $z_1<z_2<\dots<z_{|Z|}$, then 
$\{\delta(z_i, z_{i+1})\colon i<|Z|\}$ is a set of size $|Z|-1$ whose pairs are monochromatic 
with respect to~$g$, for which reason 
$|Z|-1<n-1$. In other words, $h$ establishes the desired negative partition 
relation $2^m\nlra (n)^3_4$.

In the light of this reasoning it suffices to construct, in the same setup, an ordering 
$([2^m]^{(2)}, \strictif)$ such that for every triple $e\in [2^m]^{(3)}$ the colour $h(e)$ 
from the foregoing paragraph determines the ``isomorphism type'' of $(e^{(2)}, \strictif)$ and 
vice versa. Let $xyz$ with $1\le x<y<z\le 2^m$ be an arbitrary triple and set $\xi=\delta(xy)$
as well as  $\eta=\delta(yz)$. Our goal when constructing $\strictif$ is that 
 
\smallskip
\begin{center}
\begin{tabular}{c|c|c}
if $xyz$ is a & and & then  \\ \hline 
\rule{0pt}{3ex}  left tree  & $g(\eta, \xi)=+1$ & $xy\strictif xz\strictif yz$,\\
\rule{0pt}{3ex}  left tree  & $g(\eta, \xi)=-1$ & $xy\strictif yz\strictif xz$, \\
\rule{0pt}{3ex}  right tree & $g(\xi, \eta)=+1$ & $yz\strictif xy\strictif xz$, \\
\rule{0pt}{3ex}  right tree & $g(\xi, \eta)=-1$ & $yz\strictif xz\strictif xy$.
\end{tabular}
\end{center}
\medskip

If this can be accomplished, then every set $Z\subseteq [2^m]$, for which $(Z^{(2)}, \strictif)$
is ordered canonically, is monochromatic with respect to the colouring $h$ and, therefore, the 
proof of $\LL^{(2)}(n)>2^m$ will be complete. 

Let us introduce one final piece of notation.
Given a vector $x=(x_1, \dots, x_m)\in \{-1, +1\}^m$ and~$\delta\in [m]$ we 
define the vector $x\ltimes\delta\in\{-1, +1\}^{m-\delta}$ by
\[
	x\ltimes\delta=\bigl(x_{\delta+1}g(\delta, \delta+1), \dots, x_mg(\delta, m)\bigr)\,.
\]

It is easy to see that there exists an ordering $\strictif$ on $[2^m]^{(2)}$ such 
that for any two distinct 
pairs $xy$ and $x'y'$ with $x<y$ and $x'<y'$ the following statements hold.  
\begin{enumerate}[label=\nlabel]
	\item\label{it:1} If $\delta(xy)\ne \delta(x'y')$, then
		\[
			xy\strictif x'y'
			\,\,\, \Longleftrightarrow \,\,\,
			\delta(x'y') < \delta(xy)\,.
		\]
	\item\label{it:2} If $\delta=\delta(xy)= \delta(x'y')$ and at least one 
			of $x\ltimes \delta \ne x'\ltimes \delta$ or $y\ltimes \delta \ne y'\ltimes \delta$
			holds, then  
		\[
			xy\strictif x'y'
			\,\,\, \Longleftrightarrow \,\,\,
			(x\ltimes \delta)\circ (y\ltimes \delta)
			\llex
			(x'\ltimes \delta)\circ (y'\ltimes \delta)\,,
		\]
		where $\circ$ indicates concatenation, i.e., that we are comparing $(2m-2\delta)$-tuples
		lexicographically.
\end{enumerate}

In fact, these two properties alone define a partial order on $[2^m]^{(2)}$, which 
can be extended arbitrarily to a linear order~$\strictif$. 
Note that the condition~$x\ltimes \delta \ne x'\ltimes \delta$ occurring in~\ref{it:2}
means that $x$ and $x'$ disagree after the $\delta^{\text{th}}$ coordinate.

It remains to show that these two properties of~$\strictif$ imply the four statements in the 
above table.
Let us return to this end to our triple~$xyz$ with $1\le x<y<z\le 2^m$ 
and $\xi=\delta(xy)$, $\eta=\delta(yz)$.
If $xyz$ is a left tree, then $\delta(xz)=\delta(yz)<\delta(xy)$, for which reason~\ref{it:1} 
implies $xy\strictif xz, yz$ and it remains to compare the latter two pairs. 
Since $x\ltimes \eta\ne y\ltimes \eta$, the second rule declares that $xz\strictif yz$ holds if 
and only if $x\ltimes \eta\llex y\ltimes \eta$.
Now $x$ and $y$ agree in the first $\xi-1$ coordinates and continue with $x_\xi=-1$, $y_\xi=+1$.
Thus $x\ltimes \eta$ and $y\ltimes \eta$ agree in their first $\xi-\eta-1$ coordinates and 
proceed with $x_\xi g(\eta, \xi)$, $y_\xi g(\eta, \xi)$, respectively. So altogether $xz\strictif yz$
is equivalent to $g(\eta, \xi)=+1$ and we have confirmed the two upper rows of the table. 
The case where $xyz$ is a right tree is discussed similarly. Summarising, we have thereby proved
the first assertion in Theorem~\ref{thm:lower}.

\subsection{Negative stepping-up lemma for Leeb numbers}

Throughout this subsection we fix two integers $k\ge 3$ and $n\ge k$. We set $m=\LL^{(k-1)}(n)-1$
and study orderings on $[2^m]^{(k)}$, where~$2^m$ is again identified with $\{+1, -1\}^m$. 
We are to exhibit an ordering $([2^m]^{(k)}, \strictif)$ without canonically ordered subsets 
of size $4n+k$. Preparing ourselves we consider the following concept. 

\begin{definition}\label{dfn:imp}
	A set $Z\subseteq [2^m]$ is said to be {\it impeded} (see Figure~\ref{fig:impede}) if there exist 
	two elements $x<y$ in~$Z$
	with the following properties.
	\begin{enumerate}[label=\alabel]
		\item\label{it:impa} At least $k$ members of $Z$ lie between $x$ and $y$.
		\item\label{it:impb} Setting $\delta=\delta(x, y)$ there exist $x', x'', y', y''\in Z$
			such that $x<x'<x''<y''<y'<y$ and $x''_\delta=-1$ while $y''_\delta=+1$.
	\end{enumerate}
\end{definition}

\begin{figure}[h]
\centering
{\hfil \begin{tikzpicture}[scale=0.9]
    
    \draw (0,0) ellipse (2cm and 0.7cm);
    \draw (6,0) ellipse (2cm and 0.7cm);
    \coordinate [label=above:$x$] (x) at (-1.1,0);
    \coordinate [label=above:$x'$] (x') at (0,0);
    \coordinate [label=above:$x''$] (x'') at (1.1,0);
    \coordinate [label=above:$y$] (y) at (7.1,0);
    \coordinate [label=above:$y'$] (y') at (6,0);
    \coordinate [label=above:$y''$] (y'') at (4.9,0);
    \coordinate [label=left:$\delta+1$] (l1) at (-4,-3);
    \coordinate [label=left:$\delta$] (l2) at (-4,-4);
    
    \filldraw (0,-0.1) circle (1.5pt);
    \filldraw (-1.1,-0.1) circle (1.5pt);
    \filldraw (1.1,-0.1) circle (1.5pt);
    \filldraw (7.1,-0.1) circle (1.5pt);
    \filldraw (6,-0.1) circle (1.5pt);
    \filldraw (4.9,-0.1) circle (1.5pt);
    \draw (-4,-3) circle (1.5pt);
    \draw (-4,-4) circle (1.5pt);

    \coordinate (r) at (0,-3);
    \coordinate (s) at (6,-3);
    \coordinate (t) at (3,-4);
    \coordinate (a1) at (-3,-1);
    \coordinate (a2) at (-3,1);
    \coordinate [label=right:$Z$] (a3) at (9,1);
    \coordinate (a4) at (9,-1);
    \coordinate (e1) at (-1.95,-0.15);
    \coordinate (e2) at (1.95,-0.15);
    \coordinate (f1) at (4.05,-0.15);
    \coordinate (f2) at (7.95,-0.15);
    \coordinate (r1) at (10,-3);
    \coordinate (r2) at (10,-4);
    \coordinate (l3) at (-4,-1.8);
    \coordinate (l4) at (-4,-5);

    \draw (e1)--(r)--(e2);
    \draw (f1)--(s)--(f2);
    \draw (r)--(t)--(s);
    \draw (a1)--(a2)--(a3)--(a4)--(a1);
    \draw[dashed] (l1)--(r1);
    \draw[dashed] (l2)--(r2);
    \draw (l3)--(l4);  
    \end{tikzpicture}\hfil}
    \caption{An impeded set $Z$.}
    \label{fig:impede}
\end{figure}
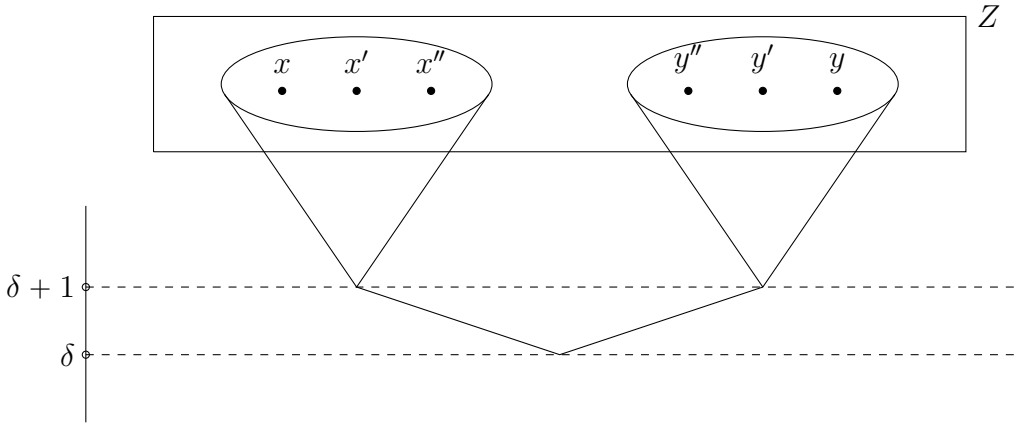

Roughly speaking, every unimpeded set contains a large comb.
  
\begin{lemma}\label{lem:2050}
	Every set $Z\subseteq [2^m]$ is either impeded or possesses two subsets $L, R\subseteq Z$
	with $|L|+|R|\ge \tfrac12(|Z|-k+1)$ forming a left comb and a right comb, respectively.  
\end{lemma}

\begin{proof}
	Otherwise let $Z$ be a minimal counterexample. Since every set with at most two elements 
	is a two-sided comb, we have $|Z|\ge k+2$. Now the vectors $x=\min(Z)$
	and $y=\max(Z)$ satisfy part~\ref{it:impa} of Definition~\ref{dfn:imp}. Let $x'$, $x''$ be the 
	second and third-smallest vector in $Z$ and, similarly, let $y'$ and $y''$ be the second and 
	third largest one. If both $x''_\delta=-1$ and $y''_\delta=+1$ would hold, 
	where $\delta=\delta(x, y)$, then the set $Z$ were impeded, contrary to its choice. 
	
	Thus at least one of the cases $x''_\delta=+1$ and $y''_\delta=-1$ occurs and for reasons 
	of symmetry we may suppose that $x''_\delta=+1$. The set $Z_\star=Z\setminus \{x, x'\}$ 
	cannot be impeded for then~$Z$ was impeded as well. So by the minimality of $Z$ there 
	exist sets $L, R_\star\subseteq Z_\star$ forming a left comb and a right comb, respectively,
	and satisfying $|L|+|R_\star|\ge \tfrac12(|Z_\star|-k+1)$. For any two elements $r<r'$
	from $R_\star$ we have $\delta(x, r)=\delta<\delta(r, r')$ and, consequently, $R=\{x\}\cup R_\star$
	is a right comb. But now $|L|+|R|=|L|+|R_\star|+1\ge \tfrac12(|Z|-k+1)$ yields a 
	final contradiction.
\end{proof}

We proceed with some further notation. Let $\llex$ denote the lexicographic ordering 
on~$[2^m]^{(k)}$. So explicitly for $X, Y\in [2^m]^{(k)}$ we have $X\llex Y$ if and only 
if $\sum_{x\in X}2^{-x}$ is larger than $\sum_{y\in Y}2^{-y}$. 

For every set $X\in [2^m]^{(k)}$ we determine a natural
number $\phi(X)$ according to the following rule: Consider $x^-=\min(X)$, 
$x^+=\max(X)$ and compute $\delta=\delta(x^-, x^+)$. Notice that all elements of $X$ agree
in their first $\delta-1$ coordinates, while in the $\delta^{\mathrm{th}}$ 
coordinate~$x^-_\delta=-1$ and~$x^+_\delta=+1$. 
Let $\phi(X)$ count the number of members of $X$ having 
a $-1$ as their~$\delta^{\mathrm{th}}$ coordinate. Notice that if $X$ is a left comb, then 
$\phi(X)=|X|-1$, while right combs satisfy $\phi(X)=1$.

Next, if $B$ is either a left or a right comb, we know 
that $\pi(B)=\bigl\{\delta(xy)\colon xy\in B^{(2)}\bigr\}\subseteq [m]$ is a set of size $|B|-1$ 
whose elements can be found by applying $\delta$ to the $|B|-1$ pairs of consecutive members
of $B$. We call $\pi(B)$ the {\it projection} of $B$.   

Finally, instead of a colouring we shall apply an appropriate ordering $([m]^{(k-1)}, \YY)$ 
for comparing projections.
Because of $m<\LL^{(k-1)}(n)$ this ordering can be chosen without canonical subsets of size $n$. 

We now state four desirable properties we would like an ordering $([2^m]^{(k)}, \strictif)$
to satisfy for any two distinct sets $X, Y\in [2^m]^{(k)}$.

\begin{enumerate}[label=\nlabel]
	\item\label{it:11} If $\phi(X)$ is odd and $\phi(Y)$ is even, then $X\strictif Y$. 
	\item\label{it:22} If $X$ and $Y$ are two combs in the same direction having different 
		projections, then $X\strictif Y \Longleftrightarrow \pi(X)\YY\pi(Y)$.
	\item\label{it:33} If $X$ and $Y$ are left combs with $\pi(X)=\pi(Y)$, 
		then $X\strictif Y \Longleftrightarrow Y\llex X$.  
	\item\label{it:44} If $X$ and $Y$ are right combs with $\pi(X)=\pi(Y)$, 
		then $X\strictif Y \Longleftrightarrow X\llex Y$.  
\end{enumerate}       

It is not hard to see that these properties do not contradict each other. Indeed, in order to 
satisfy~\ref{it:11} one partitions $[2^m]^{(k)}$ into two sets depending on the parity of $\phi$
and puts the odd class before the even class. The remaining rules~\ref{it:22}\,--\,\ref{it:44}
only concern the internal orderings within these two classes. Similarly,~\ref{it:22} tells
us to sub-partition the combs within both classes according to their projections 
and~\ref{it:33},~\ref{it:44} only give further instructions how to compare combs within the 
same sub-class.  

Now it suffices to show for no ordering $([2^m]^{(k)}, \strictif)$ 
satisfying~\ref{it:11}\,--\,\ref{it:44}
there is a set $Z\subseteq [2^m]$ of size~$4n+k$ such that $(Z^{(k)}, \strictif)$ 
is ordered canonically. Assume contrariwise that~$\strictif$ and~$Z$ have these properties. 
If $L\subseteq Z$ is a left comb, then so is every $X\in L^{(k)}$ and by~\ref{it:22} it follows 
that $\pi(L)^{(k-1)}$ is ordered canonically by $\YY$. So our choice of $\YY$ 
yields $|\pi(L)|\le n-1$ and, consequently, $|L|\le n$. For the same reason, every right 
comb $R\subseteq Z$ has at most the size~$n$. But now Lemma~\ref{lem:2050} tells us that $Z$
is impeded. Pick $x, x', x'', y'', y', y\in Z$ as in Definition~\ref{dfn:imp}. 

\smallskip

{\it \hskip2em Special case: $k=3$}

\smallskip

Recall that every triple is either a left tree or a right tree. Owing to~\ref{it:11} the 
right trees precede the left trees with respect to $\strictif$.

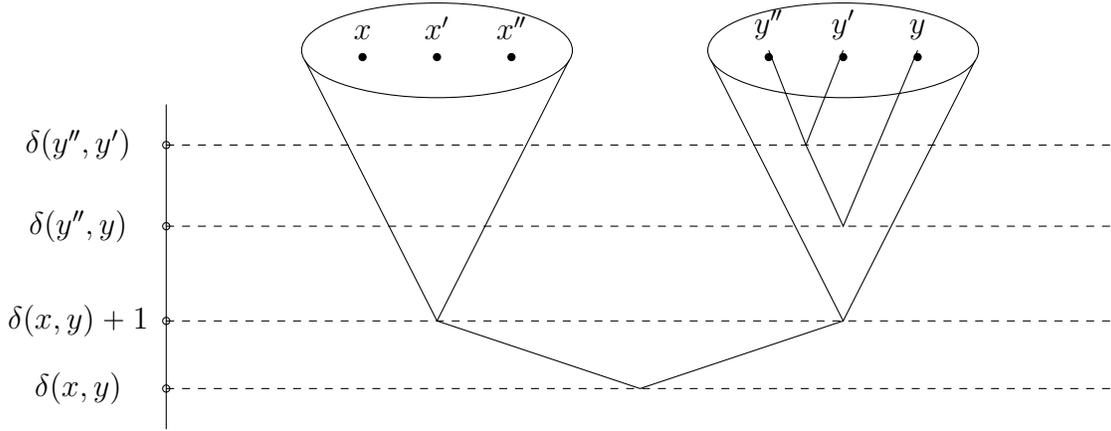
\begin{figure}[h]
\centering
{\hfil \begin{tikzpicture}[scale=0.9]
    
    \draw (0,0) ellipse (2cm and 0.7cm);
    \draw (6,0) ellipse (2cm and 0.7cm);
    \coordinate [label=above:$x$] (x) at (-1.1,0);
    \coordinate [label=above:$x'$] (x') at (0,0);
    \coordinate [label=above:$x''$] (x'') at (1.1,0);
    \coordinate [label=above:$y$] (y) at (7.1,0);
    \coordinate [label=above:$y'$] (y') at (6,0);
    \coordinate [label=above:$y''$] (y'') at (4.9,0);
    \coordinate (l1) at (-4,-4);
    \coordinate (l2) at (-4,-5);
    \coordinate (l3) at (-4,-1.4);
    \coordinate (l4) at (-4,-2.6);
    \node[] at (-5.3,-5) {$\delta(x,y)$};
    \node[] at (-5.3,-4) {$\delta(x,y)+1$};
    \node[] at (-5.3,-2.6) {$\delta(y'',y)$};
    \node[] at (-5.3,-1.4) {$\delta(y'',y')$};
    
    \filldraw (0,-0.1) circle (1.5pt);
    \filldraw (-1.1,-0.1) circle (1.5pt);
    \filldraw (1.1,-0.1) circle (1.5pt);
    \filldraw (7.1,-0.1) circle (1.5pt);
    \filldraw (6,-0.1) circle (1.5pt);
    \filldraw (4.9,-0.1) circle (1.5pt);
    \draw (-4,-4) circle (1.5pt);
    \draw (-4,-5) circle (1.5pt);
    \draw (-4,-1.4) circle (1.5pt);
    \draw (-4,-2.6) circle (1.5pt);

    \coordinate (r) at (0,-4);
    \coordinate (s) at (6,-4);
    \coordinate (t) at (3,-5);
    \coordinate (w) at (5.45,-1.4);
    \coordinate (v) at (6,-2.6);
    \coordinate (e1) at (-1.95,-0.15);
    \coordinate (e2) at (1.95,-0.15);
    \coordinate (f1) at (4.05,-0.15);
    \coordinate (f2) at (7.95,-0.15);
    \coordinate (r1) at (10,-4);
    \coordinate (r2) at (10,-5);
    \coordinate (r3) at (10,-1.4);
    \coordinate (r4) at (10,-2.6);
    \coordinate (l5) at (-4,-0.8);
    \coordinate (l6) at (-4,-5.6);

    \draw (e1)--(r)--(e2);
    \draw (f1)--(s)--(f2);
    \draw (r)--(t)--(s);
    \draw[dashed] (l1)--(r1);
    \draw[dashed] (l2)--(r2);
    \draw[dashed] (l3)--(r3);
    \draw[dashed] (l4)--(r4);
    \draw (l5)--(l6);
    \draw (y'')--(w)--(y');
    \draw (w)--(v)--(y);

    \end{tikzpicture}\hfil}
    \caption{The case $k=3$.}
\end{figure}

Now consider the quintuples $A=\{x, x', x'', y', y\}$ and $B=\{x, x', y'', y', y\}$,
whose elements have just been listed in increasing order. By canonicity we know 
that $(A^{(3)}, \strictif)$ and $(B^{(3)}, \strictif)$ are isomorphic via the map that 
sends $x''$ to $y''$ and keeps the four other elements fixed. Since $xx''y'$ and $xx''y$
are left trees with the same projection,~\ref{it:33} yields $xx''y\strictif xx''y'$.
This statement in $(A^{(3)}, \strictif)$ translates to $xy''y\strictif xy''y'$. 
As we are now comparing right trees,~\ref{it:44} discloses that $xy''y$ and $xy''y'$
cannot have the same projection, for which reason $\delta(y''y)\ne \delta(y''y')$.
In other words, $y''y'y$ is a left tree and $\delta(y''y) = \delta(y'y)$. 
It follows that the right trees $xy''y$ and $xy'y$ have the same projection, 
so another application of~\ref{it:44} leads to $xy''y\strictif xy'y$. When transferred
back to $A$ this statement becomes $xx''y\strictif xy'y$, so we have found a left tree 
coming before a right tree, contrary to~\ref{it:11}. This concludes the discussion of the 
case $k=3$.

\smallskip

{\it \hskip2em General case: $k\ge 4$}

\smallskip

We may assume that $x'$, $x''$ are the elements of $Z$ immediately succeeding $x$ and,
similarly, that $y''$, $y'$ are immediately preceeding $y$. According to part~\ref{it:impa}
of Definition~\ref{dfn:imp} there is a set $W\subseteq Z$ with $|W|=k-2$, $\min(W)=x''$,
and $\max(W)=y''$. Consider the sets $A=W\cup\{x, x', y\}$ and $B=W\cup\{x, y', y\}$
of size $k+1$. By canonicity we know that the orderings $(A^{(k)}, \strictif)$ 
and $(B^{(k)}, \strictif)$ are in the following sense essentially the same. For~$i\in [k+1]$
let $A_i$ and $B_i$ denote the sets of size $k$ arising from $A$ and $B$ by removing 
their~$i^{\mathrm{th}}$ elements, respectively. Now if $A_i\strictif A_j$ holds for 
some $i, j\in [k+1]$,
then $B_i\strictif B_j$ follows. Set $\phi=\phi(A)$ and observe that $3\le \phi\le k-1$
due to the presence of $x$, $x'$, $x''$, $y$, and $y''$. Moreover 
\begin{align*}
	\phi(A_1)=\dots =\phi(A_\phi)=\phi-1, \qquad \phi(A_{\phi+1})=\dots = \phi(A_{k+1})=\phi, \\
	\phi(B_1)=\dots =\phi(B_{\phi-1})=\phi-2, \quad \phi(B_{\phi})=\dots = \phi(B_{k+1})=\phi-1\,.
\end{align*}
In view of~\ref{it:11}, this leads to the equivalences
\[
	A_1\strictif A_{k+1}
	\,\,\,\Longleftrightarrow\,\,\,
	\phi\text{ is even }
	\,\,\,\Longleftrightarrow\,\,\,
	B_{k+1}\strictif B_{1}\,,
\]
which contradict the canonicity of $\strictif$. 
Altogether, this completes the proof of Theorem~\ref{thm:lower}. 

\section{Upper bounds: the canonisation procedure}\label{sec:upper}

\subsection{Preparation}\label{subsec:31}
The proof of Theorem~\ref{thm:upper} starts with some synchronisation principles for 
colourings $f\colon X^{(r)}\lra\Gamma$, where $X\subseteq \NN$, $r\ge 2$, and $\Gamma$ denotes a 
finite {\it set of colours}. A special r\^{o}le will be played by the set of 
colours 
\begin{equation}\label{eq:Phi}
	\Phi=\{-1, +1\}\,.
\end{equation} 

Given a colouring $f\colon X^{(r)}\lra\Phi$ we define the {\it opposite colouring}
$(-f)\colon X^{(r)}\lra\Phi$ by $(-f)(x)=-(f(x))$ for all $x\in X^{(r)}$.  
For simplicity we call a colouring $f\colon X^{(r)}\lra\Phi$,
where $r\ge 1$ and $X\subseteq \NN$, {\it increasing} if 
$f(x_1, \dots, x_{r-1}, y)\le f(x_1, \dots, x_{r-1}, y')$ holds for all $x_1<\dots <x_{r-1}<y<y'$
from $X$ and {\it decreasing} if its opposite colouring is increasing. 
We refer to $f$ as {\it monotone} if it is either increasing or decreasing. 
Two colourings $f\colon X^{(r)}\lra\Phi$ and $g\colon Y^{(s)}\lra\Phi$
are said to be {\it monotone in the same direction} if either both are increasing or both are
decreasing. Similarly, $f$ and $g$ are {\it monotone in opposite directions} if $f$ and $-g$ 
are monotone in the same direction.  

\begin{lemma}\label{lem:1730}
	Suppose that $s, t\ge 1$ are positive integers and that $X\subseteq \NN$ is a set 
	of size $|X|\ge (s-1)(t-1)+1$. If the colouring $f\colon X^{(2)}\lra \Phi$ is monotone, 
	then either there is a monochromatic set $A\in X^{(s)}$ for the colour $-1$ or a 
	monochromatic set $B\in X^{(t)}$ for the colour $+1$.
\end{lemma} 

\begin{proof}
	Due to the symmetry between $f$ and $-f$ we may suppose that $f$ is increasing. 
	We argue by induction on $t$. The base case $t=1$ offers no difficulty, for every 
	set of size one is monochromatic in both colours. Now suppose that our claim is already known
	to be true for $t-1$ in place of $t$, where $t\ge 2$. Set $z=\max(X)$
	and consider the set $N=\{x\in X\colon f(xz)=-1\}$. 
		
	\smallskip
	
	{\it \hskip 2em First Case: $|N|\ge s-1$}
	
	\smallskip
	
	Now it suffices to prove that $A=N\cup\{z\}$ is monochromatic in colour $-1$. 
	So let $x<y$ from $A$ be arbitrary and observe that $x\in N$. If $y=z$, then $f(xy)=-1$
	follows from the definition of $N$. Otherwise we have $y<z$ and the assumption that $f$ 
	is increasing implies $f(xy)\le f(xz)=-1$, whence $f(xy)=-1$. 
	
	\smallskip
	
	{\it \hskip 2em Second Case: $|N|< s-1$}
	
	\smallskip
	
	Now the set $X_\star=X\setminus (N\cup \{z\})$
	satisfies $|X_\star|\ge |X|-(s-1)\ge (s-1)(t-2)+1$. Thus the induction hypothesis applies 
	to $X_\star$ and yields either a monochromatic set $A$ of size $s$ for $-1$ or a 
	monochromatic set $B'$ of size $t-1$ for $+1$. In the former case we are done immediately
	and in the latter case the definition of $N$ entails that $B=B'\cup\{z\}$ is monochromatic 
	for the colour $+1$.  
\end{proof}

As observed by the referee, one can also prove this lemma by applying Dilworth's 
theorem~\cite{Dil} to the partial ordering $(X, \le^\star)$ defined by 
declaring $a\le^\star b$ if and only if $a\le b$ and $f(ab)=+1$. (Notice that the 
monotonicity of $f$ ensures that $\le^\star$ is transitive.) In any case, the lemma
iterates as follows.

\begin{lemma}\label{lem:1731}
	Given a natural number $n$ and a nonempty finite index set $I$ let $X\subseteq \NN$
	be a set with $|X|\ge |I|!n^{|I|+1}$. If $(f_i)_{i\in I}$ denotes a family of 
	monotone colourings $f_i\colon X^{(2)}\lra \Phi$, then there is a set $Z\subseteq X$
	of size $n$ which is monochromatic with respect to every $f_i$.
\end{lemma}

\begin{proof}
	Replacing some of the colourings $f_i$ by their opposites, if necessary, we may assume 
	that all of them are increasing. We argue by induction on $|I|$, the base case $|I|=1$ 
	being dealt with in the foregoing lemma. Now suppose that $|I|\ge 2$ and that 
	the lemma has already been established for every set $I'\subsetneq I$ instead of $I$.
	Given a set $X$ and a family of increasing colourings $(f_i)_{i\in I}$ of the pairs of 
	elements from $X$ we define an auxiliary colouring $f\colon X^{(2)}\lra \Phi$ by 
	$f=\min\{f_i\colon i\in I\}$, i.e., 
	\[
		f(xy)=+1 
		\,\,\, \Longleftrightarrow \,\,\,
		\forall i\in I\,\, f_i(xy)=+1\,.
	\]

	Evidently, $f$ is increasing as well. So by Lemma~\ref{lem:1730} there either exists 
	a set $A\subseteq X$ with $|A|\ge |I|!n^{|I|}$ that is monochromatic for $f$ in colour $-1$
	or there is a set $Z\subseteq X$ of size $n$ that is monochromatic for $f$ in colour $+1$. 
	In the latter case $Z$ monochromatically has colour~$+1$ with respect to every $f_i$ 
	and thus it possesses the desired property. 
	
	So suppose from now on that the former case occurs, i.e., that we have obtained
	a set~$A$ of size $|A|\ge |I|!n^{|I|}$ that is monochromatic with respect to $f$ in colour~$-1$. 
	Put $z=\max(A)$ and $X_i=\{a\in A\colon f_i(az)=-1\}$ for every $i\in I$. 
	Since $f(az)=-1$ holds for all $a\in A\setminus\{z\}$, we know 
	that $\bigcup_{i\in I}X_i=A\setminus\{z\}$.
	Owing to the box principle (Schubfachprinzip) this 
	implies $|X_{i(\star)}|\ge (|I|-1)!n^{|I|}$ for 
	some $i(\star)\in I$. As in the first case of the proof of Lemma~\ref{lem:1730} one can 
	show that the set $X_{i(\star)}\cup\{z\}$ monochromatically receives the colour $-1$ with 
	respect to~$f_{i(\star)}$. It remains to apply the induction hypothesis to this set and 
	to $I'=I\setminus \{i(\star)\}$. 
\end{proof}

Let us now return to an arbitrary colouring $f\colon X^{(r)}\lra \Gamma$, where $X\subseteq \NN$,
$r\ge 2$, and $\Gamma$ denotes a set of colours. We say that $f$ is {\it governed} by another 
colouring $f'\colon X^{(\ell)}\lra\Gamma$, where $\ell\in [r-1]$, 
if $f(x_1, \dots, x_r)=f'(x_1, \dots, x_\ell)$
holds for all $x_1<\dots <x_r$ from $X$. 
Standard references call $f$ {\it end-homogeneous}
if it is governed by some $(r-1)$-ary function~$f'$ in this sense, but for us it will be useful 
to have words expressing the relationship between~$f$ and~$f'$. 

\begin{lemma}\label{lem:2150}
	Given integers $n\ge r\ge 2$ and $s\ge 3$, a set of colours $\Gamma$, and a nonempty finite index 
	set $I$ let $X\subseteq \NN$ be a set of 
	size $|X|\ge (2|\Gamma|)^{n^{r-1}} \cdot (2|I|n^{s-2})^{|I|n^{s-2}}$. 
	If $f\colon X^{(r)}\lra\Gamma$ is an arbitrary colouring and $(g_i)_{i\in I}$ 
	with $g_i\colon X^{(s)}\lra \Phi$ is a family of 
	monotone colourings, then there exist a set $Z\subseteq X$ of size $n$, a colouring 
	$f'\colon Z^{(r-1)}\lra \Gamma$ and a family $(g'_i)_{i\in I}$ of monotone colourings 
	$g'_i\colon Z^{(s-1)}\lra \Phi$ such that 
	\begin{enumerate}[label=\alabel]
		\item\label{it:aa} $f'$ governs the restriction of $f$ to $Z^{(r)}$,
		\item\label{it:bb} for every $i\in I$ the colouring $g'_i$ governs the restriction 
			of $g_i$ to $Z^{(s)}$,
		\item\label{it:cc} and for every $i\in I$ the colourings $g_i$, $g'_i$ are monotone in 
			opposite directions.
	\end{enumerate}
\end{lemma}     

\begin{proof}
	As usual, we may suppose that every $g_i$ is increasing. 
	An integer $\nu\in [n]$ is said to be {\it secure} if there exist 
	sets $Z_\nu, A_\nu\subseteq X$
	with $|Z_\nu|=\nu$ and $\max(Z_\nu)<\min(A_\nu)$, 
	a function $f'\colon Z_\nu^{(r-1)}\lra \Gamma$, and a family 
	$(g'_i)_{i\in I}$ of decreasing colourings $g'_i\colon Z_\nu^{(s-1)}\lra \Phi$ 
	with the following properties.
	\begin{enumerate}[label=\rmlabel]
		\item\label{it:51} If the numbers $z_1<\dots<z_{r-1}$ are in $Z_\nu$ and $z_r>z_{r-1}$ 
				is in $Z_\nu\cup A_\nu$, then $f(z_1, \dots, z_r)=f'(z_1, \dots, z_{r-1})$.
		\item\label{it:52} If $i\in I$, the numbers $z_1<\dots<z_{s-1}$ are in $Z_\nu$, 
				and $z_{s}>z_{s-1}$ is in $Z_\nu\cup A_\nu$, 
				then $g_i(z_1, \dots, z_{s})=g'_i(z_1, \dots, z_{s-1})$.
		\item\label{it:53} If $i\in I$, the numbers $z_1<\dots<z_{s-1}$ are in $Z_\nu$, 
				and $g'_i(z_1, \dots, z_{s-1})=-1$, then $g_i(z_1, \dots, z_{s-2}, y, y')=-1$
				holds for all $y, y'\in A_\nu$.
		\item\label{it:54} Finally, 
			$|A_\nu|\ge (2|\Gamma|)^{n^{r-2}(n-\nu)} \cdot (2|I|n^{s-2})^{|I|n^{s-2}-\Omega}$, 
			where $\Omega$ denotes the number of pairs $(i, Z)$ such that $i\in I$, 
			$\max(Z_\nu)\in Z\in Z_\nu^{(s-1)}$, and $g'_i(Z)=-1$.
	\end{enumerate}
	
   Observe that~\ref{it:51} and~\ref{it:52} are fairly standard conditions in arguments concerning 
   end-homogeneity. Except for the presence of $\Omega$ so is~\ref{it:54}. Clause~\ref{it:53}
   interacts well with the demand that the functions $g'_i$ be decreasing. 
  
	We contend that $1$ is secure. If $r\ge 3$ this can be seen by taking $Z_1=\{\min(X)\}$ 
	and~$A_1=X\setminus Z_1$. In the special case $r=2$ we still set~$Z_1=\{\min(X)\}$ but now 
	we need to pay some attention to condition~\ref{it:51}. The box principle 
	yields a colour~$\gamma\in \Gamma$ and a set $A_1\subseteq X\setminus Z_1$ of size 
	$|A_1|\ge (|X|-1)/|\Gamma|$ such that $f(\min(X), a_1)=\gamma$ holds for all~$a_1\in A_1$. 
	Since $A_1$ is at least as large as~\ref{it:54} demands, $1$ is indeed secure.  
	
	Notice that if $n$ turns out to be secure, then any set $Z_n$ exemplifying this 
   will have the desired properties. So it suffices to prove that the largest secure number, 
   denoted by $\nu(\star)$ in the sequel, is equal to $n$. 
  
   Assume for the sake of contradiction that $\nu(\star)<n$ and let the sets $Z_{\nu(\star)}$, 
   $A_{\nu(\star)}$, the function  $f'\colon Z_{\nu(\star)}^{(r-1)}\lra \Gamma$, and the 
   family $(g'_i)_{i\in I}$ of decreasing colourings $g'_i\colon Z_{\nu(\star)}^{(s-1)}\lra \Phi$
   witness the security of $\nu(\star)$. Moreover, let $\Omega$ denote the quantity 
   occurring in~\ref{it:54}, i.e., the number of pairs $(i, Z)$ such that $i\in I$, 
	$\max(Z_{\nu(\star)})\in Z\in Z_{\nu(\star)}^{(s-1)}$, and $g'_i(Z)=-1$.
   Due to $\Omega\le |I|\binom{\nu(\star)-1}{s-2} < |I|n^{s-2}$ 
   the second factor in~\ref{it:54} is integral and we may assume without loss of generality 
   that 
   \[
   	|A_{\nu(\star)}|
		=
		(2|\Gamma|)^{n^{r-2}(n-\nu(\star))} \cdot (2|I|n^{s-2})^{|I|n^{s-2}-\Omega}\,.
	\]

   Now the plan is to reach a contradiction by establishing the security of $\nu(\star)+1$.
   To this end we shall add a new element to $Z_{\nu(\star)}$ and extend the functions $f'$
   and $g'_i$ appropriately. The condition that seems hardest to maintain is~\ref{it:53}. 
   In fact, this task requires some preparations that have to be carried out before one decides which
   element gets added to~$Z_{\nu(\star)}$. 
    
   Set $z_{\nu(\star)}=\max(Z_{\nu(\star)})$ and define 
   \[
 		L		
		=
		\bigl\{(i, Z)\in I\times Z_{\nu(\star)}^{(s-2)}
			\colon \text{ If } z_{\nu(\star)}\not\in Z, 
				\text{ then } g'_i(Z\cup\{z_{\nu(\star)}\})=+1\bigr\}\,.
	\]
	A pair of sets $(W, D)$ with $W\subseteq L$ and
	$D\subseteq A_{\nu(\star)}$ is said to be a {\it candidate} 
	if 
	\begin{enumerate}
		\item[$\bullet$] for all $(i, Z)\in W$ and all $dd'\in D^{(2)}$ 
			we have $g_i(Z\cup\{d, d'\})=-1$, 
		\item[$\bullet$] 	and $|D|\ge (2|I|n^{s-2})^{-|W|} |A_{\nu(\star)}|$. 
	\end{enumerate}
	For instance, $(\varnothing,  A_{\nu(\star)})$ is a candidate and, consequently, we can 
	pick a candidate $(W_\star, D_\star)$ for which $|W_\star|$ is maximal. 
	Since $\Omega+|W_\star|\le \Omega+|L|\le |I|\binom{\nu(\star)}{s-2}\le |I|n^{s-2}$,
	we may again assume that 
	\[
		|D_\star|
		= 
		(2|I|n^{s-2})^{-|W_\star|} |A_{\nu(\star)}|
		=
		(2|\Gamma|)^{n^{r-2}(n-\nu(\star))} \cdot (2|I|n^{s-2})^{|I|n^{s-2}-\Omega-|W_\star|}\,.
	\]

	Now we partition $D_\star$ into two consecutive halves, i.e., we write $D_\star=D'\dcup D''$ 
	such that $|D'|=|D''|=\frac12 |D_\star|$ and $D'<D''$ (i.e., $\max(D')<\min(D'')$).
	Put $d''_\star=\min(D'')$ and consider 
	for every pair $(i, Z)\in L \setminus W_\star$ the set 
	\[
		X(i, Z)=\bigl\{d'\in D'\colon g_i(Z\cup \{d', d''_\star\})=-1\bigr\}\,.
	\]

	Assume for the sake of contradiction that there exists a pair $(i, Z)\in L\setminus W_\star$
	such that $|X(i, Z)|\ge (2|I|n^{s-2})^{-1}|D_\star|$. Now for all $y<y'$ from $X(i, Z)$
	the assumption that $g_i$ is increasing 
	yields $g_i(Z\cup \{y, y'\})\le g_i(Z\cup \{y, d''_\star\})=-1$, 
	whence $g_i(Z\cup \{y, y'\})=-1$. In other words, $\bigl(W_\star\cup\{(i, Z)\}, X(i, Z)\bigr)$
	is a candidate contradicting the maximality of $|W_\star|$. 
	
	We have thereby proved $|X(i, Z)|< (2|I|n^{s-2})^{-1}|D_\star|$ for all pairs 
	$(i, Z)\in L\setminus W_\star$. Since $|L\setminus W_\star|\le |L|\le |I|n^{s-2}$,
	this implies 
	\[
		\bigg|\bigcup_{(i, Z)\in L\setminus W_\star}X(i, Z)\bigg|<\tfrac12|D_\star|=|D'|\,.
	\]
	Therefore, we can pick an element 
	\[
		z_{\nu(\star)+1}\in D'\setminus \bigcup_{(i, Z)\in L\setminus W_\star}X(i, Z)
	\]
	and set $Z_{\nu(\star)+1}=Z_{\nu(\star)}\cup\{z_{\nu(\star)+1}\}$. 
	
	With the help of this set we intend to convince ourselves that $\nu(\star)+1$ is secure, 
	which will contradict our maximal choice of $\nu(\star)$ and thus conclude the proof 
	of Lemma~\ref{lem:2150}. Towards this goal, we need to specify an appropriate 
	set $A_{\nu(\star)+1}$, a function $f''\colon Z_{\nu(\star)+1}^{(r-1)}\lra \Gamma$, 
	and a family $(g''_i)_{i\in I}$ of decreasing 
	colourings $g''_i\colon Z_{\nu(\star)+1}^{(s-1)}\lra \Phi$ such that the 
	conditions~\ref{it:51}\,--\,\ref{it:54} hold in this setting. These choices will be made 
	in such a way that $A_{\nu(\star)+1}\subseteq D''$, the function $f''$ extends $f'$, and 
	for every $i\in I$ the colouring $g''_i$ extends $g'_i$. 
	
	Our first step consists in determining $A_{\nu(\star)+1}$ and $f''$. To this end, we define 
	for every $d''\in D''$ the function $h_{d''}\colon Z_{\nu(\star)}^{(r-2)}\lra \Gamma$ by 
	\[
		h_{d''}(Z)=f(Z\cup\{z_{\nu(\star)+1}, d''\})
		\quad \text{ for all } Z\in Z_{\nu(\star)}^{(r-2)}\,.
	\]
	Since $|Z_{\nu(\star)}^{(r-2)}|\le n^{r-2}$, there are at most $|\Gamma|^{n^{r-2}}$ possible 
	functions from $Z_{\nu(\star)}^{(r-2)}$ to $\Gamma$ and the box principle yields 
	a function $h_\star\colon Z_{\nu(\star)}^{(r-2)}\lra \Gamma$ together with a 
	set $A_{\nu(\star)+1}\subseteq D''$ satisfying $h_{d''}=h_\star$ for all $d''\in A_{\nu(\star)+1}$
	and $|A_{\nu(\star)+1}|\ge |\Gamma|^{-n^{r-2}}|D''|$. For later use we remark that 
	\[
		 |A_{\nu(\star)+1}|
		 \ge 
		 \frac{|D''|}{|\Gamma|^{n^{r-2}}}
		 =
		 \frac{|D_\star|}{2|\Gamma|^{n^{r-2}}}
		 \ge
		 (2|\Gamma|)^{n^{r-2}(n-\nu(\star)-1)} \cdot (2|I|n^{s-2})^{|I|n^{s-2}-\Omega-|W_\star|}\,.
	\]

	Now we define the extension $f''\supseteq f'$ by $f''(Z\cup \{z_{\nu(\star)+1}\})=h_\star(Z)$
	for all $Z\in Z_{\nu(\star)}^{(r-2)}$ and observe that~\ref{it:51} clearly remains valid. 
	
	Proceeding with the extensions $g''_i\supseteq g'_i$ for $i\in I$ we set 
	\[
		g''_i(Z\cup \{z_{\nu(\star)+1}\})=
		\begin{cases}
			+1 & \text{if } (i, Z)\in L\setminus W_\star \\
			-1 & \text{otherwise } 
		\end{cases}
	\]
	for all $i\in I$ and $Z\in Z_{\nu(\star)}^{(s-2)}$. Let us start the discussion of these 
	colourings by verifying that they are indeed decreasing. In other words, we want to check 
	that if $i\in I$ and $Z\in (Z_{\nu(\star)}\setminus \{z_{\nu(\star)}\})^{(s-2)}$ satisfy 
	$g''(Z\cup \{z_{\nu(\star)}\})=-1$, then $g''(Z\cup \{z_{\nu(\star)+1}\})=-1$ follows.
	To this end we just need to observe that $g''(Z\cup \{z_{\nu(\star)}\})=-1$ implies 
	$(i, Z)\not\in L$. 
	
	For the confirmation of~\ref{it:52} it suffices to argue that
	if $i\in I$, $Z\in Z_{\nu(\star)}^{(s-2)}$, and $d''\in D''$, then 
	$g_i(Z\cup\{z_{\nu(\star)+1}, d''\})=g''_i(Z\cup \{z_{\nu(\star)+1}\})$. 
	If $(i, Z)\in W_\star$, then the first bullet in the definition of $(W_\star, D_\star)$
	being a candidate yields indeed $g_i(Z\cup\{z_{\nu(\star)+1}, d''\})=-1$. 
	Next, if $(i, Z)\in L\setminus W_\star$, then $z_{\nu(\star)+1}\not\in X(i, Z)$
	entails $g_i(Z\cup \{z_{\nu(\star)+1}, d''_\star\})=+1$ and thus the assumption that $g_i$ is 
	increasing and $d''_\star=\min(D'')\le d''$ reveal $g_i(Z\cup\{z_{\nu(\star)+1}, d''\})=+1$,
	as desired. Finally, suppose that $(i, Z)\not\in L$, whence $z_{\nu(\star)}\not\in Z$
	and $g'_i(Z\cup\{z_{\nu(\star)}\})=-1$. Since $g'_i$ satisfies~\ref{it:53} 
	and $z_{\nu(\star)+1}, d''\in A_{\nu(\star)}$, we have 
	indeed $g_i(Z\cup \{z_{\nu(\star)+1}, d''\})=-1$. Altogether, we have thereby proved 
	that the colourings $g''_i$ still obey condition~\ref{it:52}.
	
	To handle the new cases of~\ref{it:53} we show that if $Z\in Z_{\nu(\star)}^{(s-2)}$ 
	and either $(i, Z)\in W_\star$ or $(i, Z)\not\in L$, then $g_i(Z\cup\{y, y'\})=-1$ holds for 
	all $yy'\in D_\star^{(2)}$. For $(i, Z)\in W_\star$ this is a direct consequence 
	of $(W_\star, D_\star)$ being a candidate, so we may suppose $(i, Z)\not\in L$ from now on,
	which, let us recall, means $z_{\nu(\star)}\not\in Z$ and $g'_i(Z\cup\{z_{\nu(\star)}\})=-1$.  
	So the desired conclusion is already covered by $g'_i$ satisfying~\ref{it:53}.
	
	Finally, we need to verify the lower bound on $|A_{\nu(\star)+1}|$ promised by~\ref{it:54}. 
	Denoting the number of pairs $(i, Z)$ such that $i\in I$, 
	$z_{\nu(\star)+1}\in Z\in Z_{\nu(\star)+1}^{(s-1)}$, and $g''_i(Z)=-1$ by $\Omega'$ it 
	suffices to check that $\Omega' = \Omega+|W_\star|$. Since $\Omega'$ counts the number 
	of pairs $(i, Z)\in I\times Z_{\nu(\star)}^{(s-2)}$ 
	satisfying $g''_i(Z\cup\{z_{\nu(\star)+1}\})=-1$, or in other words $Z\not\in L\setminus W_\star$,
	we have indeed 
	\begin{align*}
		\Omega'
		&=
		|I\times Z_{\nu(\star)}^{(s-2)}\setminus (L\setminus W_\star)| \\
		&=
		|I\times Z_{\nu(\star)}^{(s-2)}\setminus L|+|W_\star| \\
		&=
		\big|\bigl\{(i, Z)\in I\times Z_{\nu(\star)}^{(s-2)}\colon 
			z_{\nu(\star)}\not\in Z \text{ and } g'_i(Z\cup\{z_{\nu(\star)}\})=-1 \bigr\}\big|+|W_\star|\\
		&=
		\Omega+|W_\star|\,.
	\end{align*}

	This concludes the proof that $\nu(\star)+1$ is secure and, as we said before, this contradiction 
	to the choice of $\nu(\star)$ establishes Lemma~\ref{lem:2150}.   
\end{proof}

We remark that the case $I=\varnothing$ can be handled similarly but much easier. 
In this case one can ignore all parts of the proof addressing $I$ and, moreover,
there is no need to divide the set $D_\star=A_{\nu(\star)}$ into two halves. Therefore, one 
obtains the following well-known statement. 

\begin{cor}\label{cor:0053}
	If $n\ge r\ge 2$ and $N\ge |\Gamma|^{n^{r-1}}$, then for every 
	colouring $f\colon [N]^{(r)}\lra \Gamma$ there are a set $X\subseteq [N]$
	of size $n$ and a function $f'\colon X^{(r-1)}\lra \Gamma$ governing the 
	restriction of $f$ to $X^{(r)}$. \hfill $\Box$  
\end{cor}
  
\subsection{Definite orderings} \label{subsec:32}
Resuming the task of proving Theorem~\ref{thm:upper} we shall now study 
orderings $([N]^{(k)}, \strictif)$ that are ``almost canonical'' in the sense 
that we can easily extract a sign vector $\eps=(\eps_1, \dots, \eps_k)\in \{+1, -1\}^k$ and a 
permutation $\sigma\in \gS_k$ from them such that $\strictif$ behaves in some important ways 
like the canonical ordering associated with $(\eps, \sigma)$. For transparency we shall 
assume $N\ge 2k+1$ throughout this investigation. 

Observe that in the canonical case, an entry $\eps_i=+1$ in the sign vector indicates that 
we can increase a set $x\in [N]^{(k)}$ with respect to $\strictif$ by increasing its element 
in the $i^{\mathrm{th}}$ position, whereas an entry $\eps_i=-1$ signifies that it is the other 
way around. When $\strictif$ is arbitrary and $i\in [k]$ we call a 
triple $S=(X, A, Y)\in [N]^{(i-1)}\times [N]^{(2)}\times [N]^{(k-i)}$ with $X<A<Y$ 
an {\it $i$-decider}.
Notice that $i$-deciders are in natural bijective correspondence with $(k+1)$-subsets of $[N]$.
If $S=(X, A, Y)$ is such an $i$-decider and $A=\{c, d\}$ with $c<d$ we may compare the sets 
$S(c)=X\cup\{c\}\cup Y$ and $S(d)=X\cup\{d\}\cup Y$ with respect to~$\strictif$. 
If $S(c)\strictif S(d)$ we set $\eps_i(S)=+1$ and in case $S(d)\strictif S(c)$ we put $\eps_i(S)=-1$.
So roughly speaking, $\eps_i(S)=+1$ means that the $i$-decider $S$ ``thinks'' that $\eps_i$ should
have the value~$+1$. If all $i$-deciders $S$ agree and yield the same sign $\eps_i(S)$ we say 
that $\strictif$ is {\it $i$-definite} and denote the common value of all $\eps_i(S)$ by $\eps_i$. 
It may happen that $\strictif$ is $i$-definite for every $i\in [k]$ and in this case we 
call the ordering $\strictif$ {\it sign-definite} and $\eps=(\eps_1, \dots, \eps_k)$ is referred 
to as its {\it sign-vector}.
Clearly, the canonical ordering associated with a pair $(\eps, \sigma)\in \{+1, -1\}^k\times \gS_k$
is sign-definite and has sign-vector $\eps$ in this sense. 

Next we explain the extraction of a permutation from an ``almost canonical'' ordering. 
Let $ij\in [k]^{(2)}$ with $i<j$ be an arbitrary pair of indices. 
Observe that for a permutation~$\sigma\in \gS_k$, the number $i$ occurs before $j$ in the list 
$\sigma(1), \dots, \sigma(k)$ if and only if $\sigma^{-1}(i) < \sigma^{-1}(j)$. 
So, for the canonical ordering considered in Definition~\ref{def:canonical}, 
$\sigma^{-1}(i) < \sigma^{-1}(j)$ 
indicates that $i$ has in the following sense higher priority than $j$. 
If we take a set $x\in [N]^{(k)}$ and change 
both its $i^{\mathrm{th}}$ and its $j^{\mathrm{th}}$ element, while keeping the relative positions 
of the elements fixed, then whether we increase or decrease $x$ with respect to $\strictif$ is 
decided by $\eps_i$ and the direction in which we move the $i^{\mathrm{th}}$ element alone. 
Similarly $\sigma^{-1}(j) < \sigma^{-1}(i)$ means that $j$ has higher priority than $i$. 
Now suppose that the 
ordering $\strictif$ is not necessarily canonical but still sign-definite with 
sign vector $\eps=(\eps_1, \dots, \eps_k)$. By an {\it $ij$-decider}
we mean a quintuple 
\[
	P=(X, A, Y, B, Z)\in [N]^{(i-1)}\times [N]^{(2)}\times [N]^{(j-i-1)}
		\times [N]^{(2)}\times [N]^{(k-j)}
\]
with $X<A<Y<B<Z$. 
So $ij$-deciders are in natural bijective correspondence with $(k+2)$-subsets of $[N]$. 
Given such an $ij$-decider $P=(X, A, Y, B, Z)$ and two elements 
$a\in A$, $b\in B$ we set $P(a, b)=X\cup \{a\}\cup Y\cup \{b\}\cup Z$. 
The restriction of $\strictif$ to the four-element set 
\[
	\langle P\rangle=\bigl\{P(a, b)\colon (a, b)\in A\times B\bigr\}
\]
is largely determined by $\eps_i$ and $\eps_j$. 
Notably, if we write $A=\{c_i, d_i\}$ and $B=\{c_j, d_j\}$ such that $\eps_ic_i<\eps_id_i$ and 
$\eps_jc_j<\eps_jd_j$, then the $\strictif$-minimum of $\langle P\rangle$ is $P(c_i, c_j)$ and 
the $\strictif$-maximum is~$P(d_i, d_j)$. So the only piece of information we are lacking 
is how $P(c_i, d_j)$
compares to~$P(d_i, c_j)$. We let $\sigma_{ij}(P)$ denote 
\begin{center}
\begin{tabular}{c|c}
the number & provided that  \\ \hline 
\rule{0pt}{3ex}  $+1$ &  $P(c_i, d_j)\strictif P(d_i, c_j)$, \\
\rule{0pt}{3ex}  $-1$ &  $P(d_i, c_j)\strictif P(c_i, d_j)$.
\end{tabular}
\end{center}
\smallskip

So intuitively $\sigma_{ij}(P)=+1$ means that $P$ ``believes'' $i$ to have higher priority 
than $j$, while in the opposite case it is the other way around. 
If $\sigma_{ij}(P)$ stays constant as $P$ varies over 
all $ij$-deciders we say that $\strictif$ is $ij$-definite and denote the common value of 
all $\sigma_{ij}(P)$ by $\sigma_{ij}$. Finally, if $\strictif$ is {\it $ij$-definite} for all 
pairs $ij\in[k]^{(2)}$ we call $\strictif$ a {\it permutation-definite} ordering. 
We remark that in the canonical case $\sigma_{ij}=+1$ is equivalent 
to $\sigma^{-1}(i)<\sigma^{-1}(j)$ meaning that the 
statement $\sigma_{ij}\cdot\bigl(\sigma^{-1}(j)-\sigma^{-1}(i)\bigr)>0$ is always valid.
Let us check that, conversely, in the permutation-definite case we can define a permutation $\sigma$ 
having this property. 

\begin{lemma}\label{lem:1518}
	Suppose that $N\ge 2k$. If the ordering $([N]^{(k)}, \strictif)$ is both sign-definite 
	and permutation-definite, then there exists a permutation $\sigma\in \gS_k$ such that 
	\[
		\sigma_{ij}\cdot \bigl(\sigma^{-1}(j)-\sigma^{-1}(i)\bigr)>0
	\]
	holds whenever $1\le i<j\le k$.
\end{lemma}

\begin{proof}
	Let $\eps=(\eps_1, \dots, \eps_k)$ be the sign-vector of $\strictif$ and write 
	$\{2i-1, 2i\}=\{c_i, d_i\}$ such that $\eps_ic_i<\eps_id_i$ holds for every $i\in [k]$.
	Consider the sets 
	\[
		x_i=\{c_i\}\cup \{d_j\colon j\ne i\}
	\]
	and let $\sigma\in \gS_k$ be the permutation 
	satisfying 
	\[
		x_{\sigma(1)}\strictif \dots \strictif x_{\sigma(k)}\,.
	\]
	Whenever $1\le i<j\le k$ the quintuple 
	\[
		P_{ij}=\bigl(\{d_r\colon 1\le r < i\}, \{2i-1, 2i\}, \{d_s\colon i<s<j\},
			\{2j-1, 2j\}, \{d_t\colon j<t\le k\}\bigr)
	\]
	is an $ij$-decider and $\sigma_{ij}$ is the sign $+1$ if and only if $x_i\strictif x_j$,
	which is in turn equivalent to $\sigma^{-1}(i)<\sigma^{-1}(j)$. In other words, irrespective
	of whether $\sigma_{ij}$ is positive or negative the
	statement $\sigma_{ij}\cdot \bigl(\sigma^{-1}(j)-\sigma^{-1}(i)\bigr)>0$ holds.
\end{proof}

The remainder of this section deals with the question how much a sign-definite and 
permutation-definite ordering $([N]^{(k)}, \strictif)$ with sign-vector $\eps$ and 
permutation $\sigma$ (obtained by means of Lemma~\ref{lem:1518}) needs to have in 
common with the canonical ordering associated 
with $(\eps, \sigma)$ in the sense of Definition~\ref{def:canonical}. For instance, Lemma~\ref{lem:1458} below asserts that under these circumstances there
exists a dense subset $W\subseteq [N]$ such that $(W^{(k)}, \strictif)$ is ordered in this 
canonical way.

The strategy we use for this purpose is the following. Suppose that $x, y\in [N]^{(k)}$ 
are written in the form $x=\{x_1, \dots, x_k\}$ and $y=\{y_1, \dots, y_k\}$, 
where $x_1<\dots<x_k$ and $y_1<\dots<y_k$. Our goal is to prove that in many cases
we have 
\begin{equation}\label{eq:2718}
	x\strictif y 
	\,\,\, \Longleftrightarrow \,\,\,
	(\epsilon_{\sigma(1)}x_{\sigma(1)},\dots,\epsilon_{\sigma(k)}x_{\sigma(k)})
	 \llex 
	 (\epsilon_{\sigma(1)}y_{\sigma(1)},\dots,\epsilon_{\sigma(1)}y_{\sigma(k)})\,.
\end{equation}
We shall say that a pair of $k$-sets $(x, y)$ is {\it $(\eps, \sigma)$-sound} if 
it satisfies the equivalence~\eqref{eq:2718}. 
Moreover, we set $\Delta(x, y)=\{i\in[k]\colon x_i\ne y_i\}$. 
The plan towards Lemma~\ref{lem:1458} is to construct a sequence of sets
$[N]=W_1\supseteq\dots\supseteq W_k$ such that every pair $(x, y)$ of $k$-sets with
$x, y\subseteq W_t$ and $|\Delta(x, y)|\le t$ is $(\eps, \sigma)$-sound. 
The statement that follows will assist us in the step of an induction on $t$.  

\begin{lemma}\label{lem:1610}
	For $L\subseteq \ZZ$ let $(L^{(k)}, \strictif)$ be an ordering that is sign-definite 
	with sign-vector~$\eps$ and permutation-definite with permutation~$\sigma$. 
	Suppose $t\in [k-1]$ has the property that every pair $(x, y)$ of $k$-sets  
	with $x, y\subseteq L$ and $|\Delta(x, y)|\le t$ is $(\eps, \sigma)$-sound.

	Let $I\subseteq [k]$ be a set of $t+1$ indices and let $i\in [k]$ be the least index 
	with $\sigma(i)\in I$. Further, let $x=\{x_1, \dots, x_k\}$ and $y=\{y_1, \dots, y_k\}$ 
	be two $k$-element subsets of $L$ such that $\Delta(x, y)=I$ and $x_{\sigma(i)}<y_{\sigma(i)}$. 
	If $a$ denotes the least member of $L$ larger than $x_{\sigma(i)}$ and the conditions
	\begin{enumerate}
		\item[$\bullet$] $a<y_{\sigma(i)}$,
		\item[$\bullet$] $\sigma(i)=\min(I) \Longrightarrow a<\min\{x_{\sigma(i)+1}, y_{\sigma(i)+1}\}$,
		\item[$\bullet$] and 
			$\sigma(i)=\max(I) \Longrightarrow a>\max\{x_{\sigma(i)-1}, y_{\sigma(i)-1}\}$
	\end{enumerate} 
	hold, then the pair $(x, y)$ is $(\eps, \sigma)$-sound.  
\end{lemma}

\begin{proof}
	Without loss of generality we may assume $\eps_{\sigma(i)}=+1$. We are to prove 
	that $x\strictif y$ and in all cases this will be accomplished by finding a $k$-set 
	$z\subseteq L$ such that the pairs~$(x, z)$ and~$(z, y)$ are $(\eps, \sigma)$-sound
	and $x\strictif z\strictif y$ holds. 
	
	Suppose first that $\min(I)<\sigma(i)<\max(I)$. 
	Due to $a\in (x_{\sigma(i)}, y_{\sigma(i)})\subseteq (x_{\sigma(i)-1}, y_{\sigma(i)+1})$
	the set $z=\{x_1, \dots, x_{\sigma(i)-1}, a, y_{\sigma(i)+1}, \dots, y_k\}$ has $k$ 
	elements that have just been enumerated in increasing order. Moreover $\sigma(i)\ne \min(I)$
	yields $|\Delta(x, z)|\le t$ and, therefore, $(x, z)$ is indeed $(\eps, \sigma)$-sound, which
	implies due to $i=\min(\sigma^{-1}[I])$ implies $x\strictif z$. 
	Similarly, $\sigma(i)\ne \max(I)$ leads to $z\strictif y$.
	
	Next we suppose that $\sigma(i)=\min(I)$, which owing to the second bullet entails 
	\[
		a<\min\{x_{\sigma(i)+1}, y_{\sigma(i)+1}\}\,.
	\]
	Let $j\in [k]$ be the second smallest index with $\sigma(j)\in I$. 
	
	\smallskip
	\begin{center}
	\begin{tabular}{c|c|c}
	If  & and & then we set $z=$\\ \hline 
	\rule{0pt}{3ex}  $x_{\sigma(j)}< y_{\sigma(j)}$ & $\sigma(j)<\max(I)$ & 
		$\{x_1, \dots, x_{\sigma(i)-1}, a, x_{\sigma(i)+1}, \dots, x_{\sigma(j)}, 
			y_{\sigma(j)+1}, \dots, y_k\}$,\\
	\rule{0pt}{3ex}  $x_{\sigma(j)}< y_{\sigma(j)}$ & $\sigma(j)=\max(I)$ & 
		$\{x_1, \dots, x_{\sigma(i)-1}, a, x_{\sigma(i)+1}, \dots, x_{\sigma(j)-1}, 
			y_{\sigma(j)}, \dots, y_k\}$,\\
	\rule{0pt}{3ex}  $x_{\sigma(j)}> y_{\sigma(j)}$ & $\sigma(j)<\max(I)$ & 
		$\{x_1, \dots, x_{\sigma(i)-1}, a, y_{\sigma(i)+1}, \dots, y_{\sigma(j)}, 
			x_{\sigma(j)+1}, \dots, x_k\}$,\\
	\rule{0pt}{3ex}  $x_{\sigma(j)}> y_{\sigma(j)}$ & $\sigma(j)=\max(I)$ &
		$\{x_1, \dots, x_{\sigma(i)-1}, a, y_{\sigma(i)+1}, \dots, y_{\sigma(j)-1}, 
			x_{\sigma(j)}, \dots, x_k\}$.
	\end{tabular}
	\end{center}
\medskip

In all four cases we have listed the elements of $z$ in increasing order. Furthermore, 
in almost all cases we have $|\Delta(x, z)|\le t$ and, consequently, $x\strictif z$. 
In fact, the only exception occurs if we are in the second case and $t=2$. But if this happens, 
then  
\[
	P=\bigl(\{x_1, \dots, x_{\sigma(i)-1}\}, \{x_{\sigma(i)}, a\}, 
		\{x_{\sigma(i)+1}, \dots, x_{\sigma(j)-1}\}, \{x_{\sigma(j)}, y_{\sigma(j)}\}, 
		\{y_{\sigma(j)+1}, \dots, y_k\}\bigr)
\]
is an $\sigma(i)\sigma(j)$-decider. 
Since the ordering $\strictif$ is permutation-definite, we know that it orders the four-element 
set $\langle P\rangle$ ``correctly'' and, as $x$, $z$ belong to this set, $x\strictif z$ holds 
in this case as well. 

Similarly, in almost all cases we have $\Delta(y, z)\le t$ and $y\strictif z$, the only 
exception occurring if we are in the fourth case and $t=2$. Under these circumstances 
\[
	Q=\bigl(\{x_1, \dots, x_{\sigma(i)-1}\}, \{a, y_{\sigma(i)}\}, 
		\{y_{\sigma(i)+1}, \dots, y_{\sigma(j)-1}\}, \{x_{\sigma(j)}, y_{\sigma(j)}\}, 
		\{x_{\sigma(j)+1}, \dots, x_k\}\bigr)
\]
is an $\sigma(i)\sigma(j)$-decider and, as before, $y, z\in \langle Q\rangle$ implies $y\strictif z$.

This concludes our discussion of the case $\sigma(i)=\min(I)$ and it remains to deal with 
the case $\sigma(i)=\max(I)$. Here one can either perform a similar argument, or one argues
that this case reduces to the previous one by reversing the ordering of the ground set. That
is, one considers the ordering $\bigl((-L)^{(k)}, \strictif^\star\bigr)$ defined by 
\[
	x\strictif^\star y  
	\,\,\, \Longleftrightarrow \,\,\,
	(-x) \strictif (-y)\,,
\]
where $-L$ means $\{-\ell\colon \ell\in L\}$ and $-x$, $-y$ are defined similarly. 
If one replaces $\strictif$ by $\strictif^\star$, 
then $I^\star=(k+1)-I$ assumes the r\^{o}le of $I$, minima correspond to maxima and the second and 
third bullet in the assumption of our lemma are exchanged. 
\end{proof}

We proceed with the existence of ``dense'' canonical subsets promised earlier. 

\begin{lemma}\label{lem:1458}
	If $N\ge 2k$ and the ordering $([N]^{(k)}, \strictif)$ is both sign-definite 
	and permutation-definite, then there exists a set $W\subseteq [N]$ of size $|W|\ge 2^{1-k}N$
	such that $\strictif$ is canonical on~$W^{(k)}$.
\end{lemma}

\begin{proof}
	For every $t\in [k]$ we set 
	\[
		W_t=\bigl\{n\in [N]\colon n\equiv 1\pmod{2^{t-1}}\bigr\}\,.
	\]
	Clearly $|W_k|\ge 2^{1-k}N$, so it suffices to prove that $\strictif$ is canonical on~$W_k^{(k)}$.
	
	Denote the sign-vector of $\strictif$ by $\eps=(\eps_1, \dots, \eps_k)$ and let $\sigma\in \gS_k$
	be the permutation obtained from $\strictif$ by means of Lemma~\ref{lem:1518}. 
	Let us prove by induction on $t\in [k]$ that if two $k$-element subsets $x, y\subseteq W_t$
	satisfy $|\Delta(x, y)|\le t$, then $(x, y)$ is $(\eps, \sigma)$-sound. 
	
	In the base case $t=1$ we have $W_1=[N]$ and everything follows from our assumption 
	that $\strictif$ be sign-definite. In the induction step from $t\in [k-1]$ to $t+1$ 
	we appeal to Lemma~\ref{lem:1610}. Since no two elements of $W_{t+1}$ occur in consecutive 
	positions of $W_t$, the bulleted assumptions are satisfied and thus there are no problems 
	with the induction step. 
	
	Since $|\Delta(x, y)|\le k$ holds for all $k$-element subsets of $W_k$, the case $t=k$ of 
	our claim proves the lemma. 
\end{proof}

Next we characterise the canonical orderings on $([N]^{(k)}, \strictif)$
for $N\ge 2k+1$. 

\begin{lemma}\label{lem:canonical}
	If $N\ge 2k+1$ and the ordering $([N]^{(k)}, \strictif)$ is canonical, then there exists
	a pair $(\eps, \sigma)\in \{-1, +1\}^k\times \gS_k$ such that $\strictif$ is the ordering 
	associated with $(\eps, \sigma)$. 
\end{lemma}

Let us point out that the condition $N\ge 2k+1$ cannot be omitted. 
For instance, for $k=3$ and $N=6$
one can construct a counterexample by starting with an ordering associated to 
some pair $(\eps, \sigma)$ and exchanging the positions of two disjoint sets 
partitioning~$[6]$. The resulting order is always canonical and in many cases 
it is not associated to another pair~$(\eps', \sigma')$. 

\begin{proof}[Proof of Lemma~\ref{lem:canonical}]
	Let $\eps$ and $\sigma$ denote the sign vector and the permutation of $\strictif$ as introduced
	in \S\ref{subsec:32}.  
	We shall show that $\strictif$ coincides with the canonical ordering associated to $(\eps, \sigma)$.
	
	Otherwise there exists a pair $(x, y)$ of subsets of $[N]$ that fails to be $(\eps, \sigma)$-sound.
	Assume that such a pair $(x, y)$ is chosen with $t+1=|\Delta(x, y)|$ minimum. Since $\strictif$ 
	is sign-definite we know that $t\in [k-1]$. Set $I=\Delta(x, y)$ and let $i\in [k]$ be the 
	smallest index with $\sigma(i)\in I$. By symmetry we may assume that $x_{\sigma(i)}<y_{\sigma(i)}$.
	Due to $|I|\ge 2$ it cannot be the case that~$\sigma(i)$ is both the minimum and the maximum of 
	$I$ and as in the proof of Lemma~\ref{lem:1610} it suffices to study the case $\sigma(i)\ne\max(I)$. 
	
	\smallskip

	{\it \hskip2em First Case: $\max(x\cup y)<N$}

	\smallskip
 
 	Define a strictly increasing map $\phi\colon x\cup y\lra [N]$ by $\phi(z)=z$ 
	for $z\le x_{\sigma(i)}$ and $\phi(z)=z+1$ for $z > x_{\sigma(i)}$. Now by Lemma~\ref{lem:1610}
	the pair $\bigl(\phi[x], \phi[y]\bigr)$ is $(\eps, \sigma)$-sound. Moreover, the canonicity of
	$\strictif$ tells us that $\big((x\cup y)^{(k)}, \strictif\bigr)$ 
	and $\big(\phi[x\cup y]^{(k)}, \strictif\bigr)$ are isomorphic via $\phi$. For these reasons,
	the pair $(x, y)$ is $(\eps, \sigma)$-sound as well. 

	\smallskip

	{\it \hskip2em Second Case: $\max(x\cup y)=N$}

	\smallskip
  
  	Let $\eta\colon x\cup y\lra [|x\cup y|]$ be order-preserving. In view of $|x\cup y|\le 2k<N$
	we know from the first case that the pair $\bigl(\eta[x], \eta[y]\bigr)$ is $(\eps, \sigma)$-sound
	and by canonicity so is $(x, y)$.
\end{proof}

\begin{remark}\label{rem:1945}
	The results of this subsection easily yield the weaker upper 
	bound 
	\[
		\LL^{(k)}(n)
		\le 
		\mathrm{R}^{(k+2)}\left(2^{k-1}n,\tbinom{k+2}{2}!\right)
		\le 
		t_{k+1}(C_k n)\,.
	\]

	To see this, let $([N]^{(k)}, \strictif)$ be any ordering, 
	where $N=\mathrm{R}^{(k+2)}(2^{k-1}n, \binom{k+2}{2}!)$. For a set $A\in [N]^{(k+2)}$
	there are $\binom{k+2}{2}!$ possibilities how the restriction of $\strictif$ to $A^{(k+2)}$
	can look. So by the definition of Ramsey numbers there is a set $Z\subseteq [N]$ with 
	$|Z|=2^{k-1}n$ such that 
	for all $A, B\subseteq Z$ with $|A|=|B|=k+2$ the order-preserving map $\eta\colon A\lra B$
	induces an order preserving map from $(A^{(k)}, \strictif)$ to $(B^{(k)}, \strictif)$. 
	This implies, in particular, that~$\strictif$ is sign-definite and permutation-definite 
	on $Z^{(k)}$. Thus by Lemma~\ref{lem:1458} there exists a set~$X\subseteq Z$ 
	with $|X|\ge 2^{1-k}|Z|=n$ such that $\strictif$ orders $X^{(k)}$ canonically. 
\end{remark}

Now it remains to save one exponentiation, which requires a more careful reasoning.  

\subsection{Further concepts}

Throughout this subsection, we fix a natural number $k\ge 2$. 
Let~$\Gamma$ denote the set of all orderings $\YY$ that can be imposed on $[k+2]^{(k)}$. 
We regard $\Gamma$ as a set of colours having the size $|\Gamma|=\binom{k+2}{2}!$. 
As in Remark~\ref{rem:1945} we {\it associate} with every ordering $([N]^{(k)}, \strictif)$ 
the colouring $f\colon [N]^{(k+2)}\lra \Gamma$ mapping every $A\in [N]^{(k+2)}$ to the unique
ordering $\YY\in \Gamma$ with the property that $(A^{(k)}, \strictif)$ and $([k+2]^{(k)}, \YY)$
are isomorphic via the order-preserving map from $A$ to $[k+2]$. Roughly speaking, what 
Theorem~\ref{thm:upper} asserts is that colourings of this form are so special that one can 
prove a Ramsey theorem for them that requires just $k$ exponentiations rather than the 
expected $k+1$ exponentiations. 

For an integer $m\in [k+1]$ we shall say that an ordering $([N]^{(k)}, \strictif)$ 
{\it is governed by $m$-sets} if there exists a function $f^m\colon [N]^{(m)}\lra \Gamma$
governing the function $f\colon [N]^{(k+2)}\lra \Gamma$ associated with $\strictif$.

\begin{dfn}
	If $([N]^{(k)}, \strictif)$ is an ordering, $i, m\in [k]$, and $W\subseteq [N]$ 
	we say that $\eps_i^m\colon W^{(m)}\lra \Phi$ is an {\it $m$-ary sign function} 
	for $(W^{(k)}, \strictif)$ provided the following holds: If $S=(X, A, Y)$ is an 
	$i$-decider with $X\dcup A\dcup Y\subseteq W$ and $M$ denotes the set consisting of the $m$ smallest 
	elements of $X\dcup A\dcup Y$, then $\eps_i(S)=\eps^m_i(M)$.  
\end{dfn}

In other words, the existence of an $m$-ary sign function $\eps_i^m\colon W^{(m)}\lra \Phi$
means that for $i$-deciders $S=(X, A, Y)$ with $X\dcup A\dcup Y\subseteq W$ the value of 
$\eps_i(S)$ only depends on the first $m$ elements of $X\dcup A\dcup Y$, while the 
remaining $k+1-m$ elements of $X\dcup A\dcup Y$ are destitute of any effect. 

We shall frequently utilise the fact that if $m\in [2, k+1]$ and $([N]^{(k)}, \strictif)$ is 
governed by $m$-sets, then for every $i\in [k]$ there exists an $(m-1)$-ary sign function
$\eps^{m-1}_i\colon [2, N]^{(m-1)}\lra \Phi$. This is because for an $i$-decider $S=(X, A, Y)$ 
with $X\dcup A\dcup Y\subseteq [2, N]$ the value of $\eps_i(S)$ can be inferred
from the colour $f(\{1\}\dcup X\dcup A\dcup Y)$, where $f\colon [N]^{(k+2)}\lra \Gamma$
denotes the function associated with $\strictif$.

Suppose that $m\in [2, k+1]$, $W\subseteq \ZZ$ and that $(W^{(k)}, \strictif)$ denotes an 
ordering possessing for every $i\in [k]$ an $(m-1)$-ary sign 
function $\eps^{m-1}_i\colon W^{(m-1)}\lra \Phi$. If $i<j$ from $[k]$ are two indices 
and $P=(X, A, Y, B, Z)$ is an $ij$-decider, we determine the colour $\sigma^m_{ij}(P)$	
as follows: Let $M$ denote the set consisting of the $m-1$ smallest elements of  
$X\dcup A\dcup Y\dcup B\dcup Z$, enumerate $A=\{c_i, d_i\}$ as well as $B=\{c_j, d_j\}$ 
in such a way that $\eps^{m-1}_i(M)c_i<\eps^{m-1}_i(M)d_i$ 
and $\eps^{m-1}_j(M)c_j<\eps^{m-1}_j(M)d_j$, and set  

\begin{center}
\begin{tabular}{c|c}
$\sigma^m_{ij}(P)=$ & provided that  \\ \hline 
\rule{0pt}{3ex}  $+1$ &  $P(c_i, d_j)\strictif P(d_i, c_j)$, \\
\rule{0pt}{3ex}  $-1$ &  $P(d_i, c_j)\strictif P(c_i, d_j)$.
\end{tabular}
\end{center}
\smallskip


\begin{dfn}
	If $m\in [2, k+1]$, $ij\in [k]^{(2)}$, and the ordering $(W^{(k)}, \strictif)$ 
	possesses $(m-1)$-ary sign 
	functions $\eps^{m-1}_1, \dots, \eps^{m-1}_k$, then we say that 
	$\sigma^m_{ij}\colon W^{(m)}\lra \Phi$ is an {\it $m$-ary permutation colouring} if 
	$\sigma^m_{ij}(P)=\sigma^m_{ij}(M)$ holds, whenever $P=(X, A, Y, B, Z)$ is an $ij$-decider  
	and~$M$ denotes the set of the $m$ smallest elements of $X\dcup A\dcup Y\dcup B\dcup Z$.
\end{dfn}

It should be clear that if in addition to possessing the $(m-1)$-ary sign 
functions $\eps^{m-1}_1, \dots, \eps^{m-1}_k$ the ordering $(W^{(k)}, \strictif)$
is governed by $m$-sets, then all $m$-ary permutation colourings $(\sigma^m_{ij})$
with $ij\in [k]^{(2)}$ exist. Consequently, if $([N]^{(k)}, \strictif)$ is governed 
by $m$-sets, then for every pair $ij\in [k]^{(2)}$ there exists a permutation colouring 
$\sigma^m_{ij}\colon [2, N]^{(m)}\lra \Phi$.

\subsection{Plan} \label{subsec:plan}
Let us now pause to explain our strategy for proving Theorem~\ref{thm:upper}. 
Given an arbitrary ordering $([N]^{(k)}, \strictif)$ with $N=t_{k}(n^{C_k})$
we need to find a set $X\subseteq [N]$ of size $n$ such that $(X^{(k)}, \strictif)$
is ordered canonically.
Following the same strategy as in Remark~\ref{rem:1945} it would suffice to find a set 
$Y\subseteq [N]$ of size $2^{k-1}n$ which is monochromatic with respect to the 
colouring $f\colon [N]^{(k+2)}\lra \Gamma$ associated with $\strictif$. So the problem 
is that we have only~$k$ exponentiations available but hope to synchronise 
sets of size $k+2$. What happens when we nevertheless perform the standard argument producing 
upper bounds on Ramsey numbers is that we obtain 
sets $[N]=Z_k\supseteq Z_{k-1}\supseteq \dots \supseteq Z_0$
such that for every $m\in [0, k]$ the ordering~$(Z_m^{(k)}, \strictif)$ is governed by $(m+2)$-sets 
and the size of $Z_m$ is, roughly, $m$ times exponential. In particular,~$Z_0$ has roughly the size 
$n^{C_k/2}$ and $(Z_0^{(k)}, \strictif)$ is governed by pairs. We may assume further that for 
every pair~$ij\in [k]^{(2)}$ a permutation colouring $\sigma^2_{ij}\colon Z_0^{(2)}\lra \Phi$ 
exists. Moreover, it would suffice to find a set $Y\subseteq Z_0$ that is monochromatic with respect 
to each~$\sigma^2_{ij}$.
This problem still sounds like we need an exponential rather than a polynomial dependence 
between~$|Y|$ and~$|Z_0|$, but Lemma~\ref{lem:1731} suggests the following alternative: if it 
turns out that the colourings $\sigma^2_{ij}$ are monotone, then we are done (for reasons explained 
in Lemma~\ref{lem:0127}). 

Now for some of the permutation 
colourings, like $\sigma^2_{1j}$ with $j\in [2, k]$, one can indeed prove that they are always monotone
(see Lemma~\ref{lem:2222} below), but it seems unlikely that in the setting we have 
described so far one can prove that, for instance, $\sigma^2_{57}$ is always monotone. To address
this issue we exploit that the colourings $\sigma^2_{ij}$ become known ``gradually'' in the course 
of the argument. When we have just constructed a set $Z_m$ we already know a colouring $\sigma^{m+2}_{ij}$
of its $(m+2)$-subsets such that, later, $\sigma^2_{ij}$ will govern its restriction to~$Z_0^{(m+2)}$.
Suppose that in the sequence of colourings $\sigma^k_{57}, \sigma^{k-1}_{57}, \dots, \sigma^2_{57}$ 
there is at least one that is monotone, say $\sigma^{m+2}_{57}$, where $m\ge 1$. 
What Lemma~\ref{lem:2150} tells us is that when selecting $Z_{m-1}\subseteq Z_{m}$ we can not only 
ensure that $(Z_{m-1}^{(k)}, \strictif)$ is governed by $(m+1)$-sets, but also that $\sigma^{m+1}_{57}$
is still monotone. Iterating this observation we see that if some $\sigma^m_{57}$ is monotone, then
we can protect ourselves against ever losing the monotonicity. Rather, we can arrange that all 
of $\sigma^m_{57}, \sigma^{m-1}_{57}, \dots, \sigma^2_{57}$ are monotone. We shall see in 
Lemma~\ref{lem:2240} that for~$i\ge 2$ the permutation colouring~$\sigma^i_{ij}$ is automatically 
monotone. 
For these reasons, we can indeed find a set $Z_0$ of polynomial size on which all pair colourings 
$\sigma^2_{ij}$ are monotone, which essentially proves Theorem~\ref{thm:upper}. 

\subsection{Canonisation}
Coming to the details we begin by establishing the monotonicity of~$\sigma^2_{1j}$ for~$j\in [2, k]$.   

\begin{lemma}\label{lem:2222}
	If the ordering $([N]^{(k)}, \strictif)$ is sign-definite and governed by pairs, then for 
	every $j\in [2, k]$ the colouring $\sigma^2_{1j}$ is increasing on $[2, N-k]^{(2)}$. 
\end{lemma}

\begin{proof}
	Given three numbers $r<s<t$ from $[2, N-k]$ satisfying~$\sigma_{1j}^2(rs)=+1$ we are to prove 
	that $\sigma_{1j}^2(rt)=+1$.  
	Let $\eps=(\eps_1, \dots, \eps_k)$ be the sign vector of~$\strictif$. We partition 
	\[
		[N-k+1, N]=Y\dcup B\dcup Z \quad \text{ such that } Y<B<Z
	\]
	and $|Y|=j-2$,~$|B|=2$,~$|Z|=k-j$
	and we enumerate $B=\{c_j, d_j\}$ such that $\eps_jc_j<\eps_jd_j$. 
	Notice that $P=(\varnothing, \{r, s\}, Y, B, Z)$ and $Q=(\varnothing, \{r, t\}, Y, B, Z)$ 
	are $1j$-deciders, whence $\sigma^2_{1j}(rs)=\sigma_{1j}(P)$ and~$\sigma^2_{1j}(rt)=\sigma_{1j}(Q)$.
	Thus our task is to derive $\sigma_{1j}(Q)=+1$ from $\sigma_{1j}(P)=+1$.
	
	\smallskip
	
	{\it \hskip 2em First Case: $\eps_1=+1$}
	
	\smallskip

	Now $\sigma_{1j}(P)=+1$ is defined to mean $rYd_jZ\strictif sYc_jZ$, where $rYd_jZ$
	abbreviates $\{r\}\cup Y\cup \{d_j\}\cup Z$ and $sYc_jZ$ is defined similarly. 
	Moreover, $S=(\{s, t\}, Yc_jZ)$ is a $1$-decider and $\eps_1(S)=\eps_1=+1$ 
	implies $sYc_jZ\strictif tYc_jZ$.
	So the transitivity of $\strictif$ leads to $rYd_jZ\strictif tYc_jZ$, for which reason
	we have indeed $\sigma_{1j}(Q)=+1$. 
	   	
	\smallskip
	
	{\it \hskip 2em Second Case: $\eps_1=-1$}
	
	\smallskip
	
	This time $\sigma_{ij}(P)=+1$ translates into $sYd_jZ\strictif rYc_jZ$
	and working with the $1$-decider $(\{s, t\}, Yd_jZ)$ we infer $tYd_jZ\strictif sYd_jZ$.
	Hence $tYd_jZ\strictif rYc_jZ$ and $\sigma_{1j}(Q)=+1$.
\end{proof}

Virtually the same proof also yields the following result.

\begin{lemma}\label{lem:2240}
	If $2\le i<j\le k$ and the ordering $([N]^{(k)}, \strictif)$ is governed by $i$-sets, 
	then the restriction of $\sigma^i_{ij}$ to $[2, N-k]^{(i)}$ is decreasing.
\end{lemma}

\begin{proof}
	Suppose that $X\subseteq [2, N-k]$ and $r, s\in [2, N-k]$ satisfy $|X|=i-1$, $X<r<s$, as well as 
	$\sigma^i_{ij}(Xr)=-1$. We are to prove that $\sigma^i_{ij}(Xs)=-1$. To this end we set 
	$t=N-k+1$ and take three sets $Y, Z, B\subseteq [N-k+2, N]$ with $Y<B<Z$ of the sizes
	$|Y|=j-i-1$, $|B|=2$, and $|Z|=k-j$.
	Now $P=(X, \{r, t\}, Y, B, Z)$ and $Q=(X, \{s, t\}, Y, B, Z)$ are two $ij$-deciders, we have 
	$\sigma_{ij}(P)=\sigma^i_{ij}(Xr)=-1$, and it remains to show $\sigma_{ij}(Q)=-1$. Recall 
	that~$\strictif$ being governed by $i$-sets implies, in particular, that there are $(i-1)$-ary sign 
	functions~$\eps^{i-1}_i$ and~$\eps^{i-1}_j$ defined on $[2, N]^{(k-1)}$. 
	So we can enumerate $B=\{c_j, d_j\}$ in such a way that 
	$\eps^{i-1}_j(X)c_j <\eps^{i-1}_j(X)d_j$.
	
	\smallskip
	
	{\it \hskip 2em First Case: $\eps^{i-1}_i(X)=+1$}
	
	\smallskip
	
	Our assumption $\sigma^i_{ij}(P)=-1$ is defined to mean $XtYc_jZ\strictif XrYd_jZ$. 
	Furthermore, the triple $S=(X, \{r, s\}, Yd_jZ)$ is an $i$-decider 
	and $\eps_i(S)=\eps^{i-1}_i(X)=+1$
	implies that $XrYd_jZ \strictif XsYd_jZ$. So altogether we have $XtYc_jZ\strictif XsYd_jZ$
	or, in other words, $\sigma^i_{ij}(Q)=-1$. 
	
	\smallskip
	
	{\it \hskip 2em Second Case: $\eps^{i-1}_i(X)=-1$}
	
	\smallskip

	 This time $\sigma^i_{ij}(P)=-1$ abbreviates the statement $XrYc_jZ\strictif XtYd_jZ$
	 and the case hypothesis $\eps^{i-1}_i(X)=-1$ implies $XsYc_jZ\strictif XrYc_jZ$.
	 For these reasons we have $XsYc_jZ\strictif XtYd_jZ$, i.e., $\sigma^i_{ij}(Q)=-1$.
\end{proof}

Next we elaborate on the reason why we care about the monotonicity 
of the permutation colourings. 

\begin{lemma}\label{lem:0127}
	There exists an absolute constant $C'_k$ such that if $n\ge 2k$, $N\ge C'_kn^{k^2}$, 
	the ordering $([N]^{(k)}, \strictif)$ is governed by pairs, and all permutation colourings 
	$\sigma^2_{ij}$ with $i\ge 2$ are monotone, then there is a set $Z\subseteq [N]$ of size $n$ such that 
	$(Z^{(k)}, \strictif)$ is ordered canonically.
\end{lemma}

\begin{proof}
	Since the ordering $\strictif$ is governed by pairs, it has unary sign 
	functions $\eps_1, \dots, \eps_k$ defined on $[2, N]^{(1)}$. 
	In other words, every $x\in [2, N]$ has its own ``opinion''
	\[
		\eps(x)=\bigl(\eps_1(x), \dots, \eps_k(x)\bigr)\in \Phi^k
	\]
	what the sign vector should be,
	where $\Phi=\{-1, +1\}$ was introduced in~\eqref{eq:Phi}. 
	Due to the box principle, there are a set $Y\subseteq [2, N]$ of size $|Y|\ge 2^{-k}N$
	as well as a vector $\eps\in\Phi^k$ such that $\eps(y)=\eps$ holds for every $y\in Y$. 
	Now~$(Y^{(k)}, \strictif)$ is sign-definite and by Lemma~\ref{lem:2222} the set $Y_\star\subseteq Y$
	obtained from $Y$ by deleting the smallest and the~$k$ largest elements has the following 
	property: for every 
	pair $ij\in [k]^{(2)}$ the restriction of~$\sigma^2_{ij}$ to $Y_\star^{(2)}$ is monotone. 
	Notice that~$|Y_\star|\ge 2^{-k-1}N$. 
	
	Now Lemma~\ref{lem:1731} applied to $I=[k]^{(2)}$ and 
	the family~$(\sigma^2_{ij})_{ij\in I}$
	yields a set $Z_\star\subseteq Y_\star$ of size $|Z_\star|=2^{k-1}n$ that is monochromatic with
	respect to every~$\sigma^2_{ij}$. This means that $(Z_\star^{(k)}, \strictif)$ is not only 
	sign-definite but also permutation-definite, and Lemma~\ref{lem:1458} leads to the desired set $Z$.   
\end{proof}

It remains to execute the plan laid out in \S\ref{subsec:plan}.

\begin{proof}[Proof of Theorem~\ref{thm:upper}]
	Given $k\ge 2$ we let $C'_k$ denote the constant delivered by the previous lemma.
	Recall that $\Gamma$ denotes the set of linear orderings of $[k+2]^{(k)}$. 
	Without loss of generality we can assume that $n\ge 2k$ (the remaining values
	of $n$ can then be taken care of by adjusting the constant $C_k$). 
	Now we define recursively 
	\begin{align*}
		N_0&=C'_k n^{k^2}+k+1\,, \\
		N_{m}&= (2|\Gamma|)^{N_{m-1}^{m+1}} \cdot (2k^2N_{m-1}^{m})^{k^2N_{m-1}^{m}}+k+1 
			\text{ for }m\in [k-2] \,, \\
		N_{k-1}&= |\Gamma|^{N_{k-2}^{k-2}}\,,\\
		\text{ and finally } N_k&= |\Gamma|^{N_{k-1}^{k-1}}\,.
	\end{align*}

	An easy induction on $m$ discloses that for every $m\in [0, k]$ there exists an absolute 
	constant $C_{k, m}$ such that $N_m\le t_m\bigl(C_{k, m}n^{2k^2}\bigr)$ holds for all $n\ge 2k$. 
	In particular, there exists an absolute constant $C_k$ such that independently of $n$ 
	we have $N_k\le t_{k}(n^{C_k})$.
	So it suffices to show that for every 
	ordering $([N_k]^{(k)}, \strictif)$ there exists a set $Z\subseteq [N_k]$ of size $n$ 
	such that $\strictif$ orders $Z^{(k)}$ canonically. 
	Two successive applications of Corollary~\ref{cor:0053} yield a set $Z_{k-2}\subseteq [N]$ of 
	size $|Z_{k-2}|=N_{k-2}$ such that $({Z}_{k-2}^{(k)}, \strictif)$ is governed by $k$-sets. 
	
	\begin{claim}\label{clm:1632}
		For every $m\in [0, k-2]$ there exists a set $Z_m\subseteq [N_k]$ of size $|Z_m|=N_m-k-1$
		such that $(Z_{m}^{(k)}, \strictif)$ is governed by $(m+2)$-sets and 
		for every $ij\in [m+2, k]^{(2)}$ there is a monotone permutation 
		colouring $\sigma^{m+2}_{ij}\colon Z_m^{(m+2)}\lra\Phi$.
	\end{claim}
	
	
	\begin{proof}[Proof of Claim~\ref{clm:1632}]
		We argue by decreasing induction on $m$, the base case $m=k-2$ already being known. 
		Now suppose that $m\in [k-2]$ and that we already have obtained a set~$Z_m\subseteq [N_k]$ with 
		the desired 
		property. 
		
		We plan to apply Lemma~\ref{lem:2150}
		to $N_{m-1}$, $m+2$, $m+2$, and $Z_m$ here in place of~$n$,~$r$,~$s$, and~$X$ there.  
		Instead of~$f$ there we take the function $f^{m+2}\colon Z_m^{(m+2)}\lra \Gamma$ 
		governing the function $f\colon [N_k]^{(k+2)}\lra \Gamma$ associated with $\strictif$. 
		Finally, we take the index set $I=[m+2, k]^{(2)}$,
		which clearly 
		satisfies $|I|\le k^2$, and the family $(\sigma^{m+2}_{ij})_{ij\in I}$ of monotone 
		permutation colourings $\sigma^{m+2}_{ij}\colon Z_m^{(m+2)}\lra \Phi$ provided by the 
		induction hypothesis. 
		Since 
		\[
			N_m-k-1
			\ge
			(2|\Gamma|)^{N_{m-1}^{m+1}} \cdot (2|I|N_{m-1}^{m})^{|I|N_{m-1}^{m}}\,,
		\]
		the intended application of Lemma~\ref{lem:2150} is justified and we can obtain a set 
		$Z'_{m-1}\subseteq Z_m$ of size $|Z'_{m-1}|=N_{m-1}$ 
		such that $\bigl({Z'}_{m-1}^{(k)}, \strictif\bigr)$ 
		is governed by $(m+1)$-sets and the permutation colourings $(\sigma^{m+1}_{ij})_{ij\in I}$ 
		are monotone. 
		The set $Z_{m-1}$ arising from $Z'_{m-1}$ by cutting off the smallest and the~$k$ 
		largest elements has all 
		desired properties, because by Lemma~\ref{lem:2240} the permutation 
		colourings $\sigma^{m+1}_{m+1, j}$
		with $j\in [m+2, k]$ are monotone as well. 
	\end{proof}
	
	Observe that the case $m=0$ yields a set $Z_0$ of size $|Z_0|=N_0-(k+1)=C'_k n^{k^2}$. 
	By our choice of $C'_k$ according to Lemma~\ref{lem:0127} there exists a set $Z\subseteq Z_0$
	of size $|Z|=n$ such that $(Z^{(k)}, \strictif)$ is ordered canonically.  
	This concludes the proof of Theorem~\ref{thm:upper}. 
\end{proof}

\section{Shelah's upper bound on Erd\H{o}s-Rado numbers}\label{sec:ERS}
In this section we show that the results in~\S\ref{subsec:31} can also be used for reproving 
Shelah's theorem that $k$-uniform Erd\H{o}s-Rado numbers grow (at most) $k-1$ times exponentially. 
More precisely, we shall show $\ER^{(k)}(n)\le t_{k-1}(C_k n^{6k})$, where, let us recall, 
an estimate of the form $\ER^{(k)}(n)\le N$ means that for every equivalence relation~$\equiv$ 
on~$[N]^{(k)}$ there exist an $n$-element subset $X\subseteq [N]$ and a set $I\subseteq [k]$ 
such that $\equiv$ and $\equiv_I$ agree on $X^{(k)}$. Here, for two sets of positive integers 
$x=\{x_1, \dots, x_k\}$ and $y=\{y_1, \dots, y_k\}$ with $x_1<\dots<x_k$ 
and $y_1<\dots<y_k$ and for a subset 
$I\subseteq [k]$ the statement $x\equiv_I y$ is defined to mean that~$x_i=y_i$ holds for all $i\in I$. 
Throughout the entire section we fix an integer $k\ge 2$. Moreover, $\Gamma$ will always denote the 
set of all equivalence relations on $[k+2]^{(k)}$.

\subsection{Preparation} \label{subsec:preparation}

Beginning as in Remark~\ref{rem:1945} we {\it associate} with every equivalence relation 
$([N]^{(k)}, \equiv)$ the colouring $f\colon [N]^{(k+2)}\lra \Gamma$ mapping 
every set $A\in [N]^{(k+2)}$ to the unique equivalence relation $\sim\in \Gamma$ with the 
property that $(A^{(k)}, \equiv)$ and $([k+2]^{(k)}, \sim)$ are isomorphic via the order-preserving 
map from $A$ to $[k+2]$. For $m\in [k+2]$ we say that $([N]^{(k)},\equiv)$ is {\it governed 
by $m$-sets} if there exists a function $f^m \colon [N]^{(m)} \lra \Gamma$ which governs $f$. 
For instance, every equivalence relation is governed by $(k+2)$-sets. Next we associate with every 
equivalence relation some auxiliary colourings $g_i$ using the set of colours $\Phi=\{-1, +1\}$
introduced in~\eqref{eq:Phi}.

\begin{definition} \label{dfn:pdec}
	Let $N\ge k+1$ be an integer and let $i \in [k]$ be an index. 
	The {\it $i^{\mathrm{th}}$ auxiliary function} $g_i \colon [N]^{(k+1)}\lra \Phi$ 
	of an equivalence relation $([N]^{(k)}, \equiv)$ is defined by 
	\[
		g_i(z_1, \dots, z_{k+1})= 1 
		\,\,\,\Longleftrightarrow\,\,\,
		\{z_1, \dots, z_{i-1}, z_{i+1}, \dots, z_{k+1}\} 
		\equiv 
		\{z_1, \dots, z_{i}, z_{i+2}, \dots, z_{k+1}\}
	\]
	whenever $1\le z_1<\dots<z_{k+1}\le N$.
\end{definition}

Like the permutation colourings in the previous section, these functions will later be synchronised 
by means of an iterative application of Lemma~\ref{lem:2150}. Notice that if an equivalence relation 
$([N]^{(k)},\equiv)$ is governed by $m$-sets, then there exists a function 
$g^{m-1}_i \colon [2, N]^{(m-1)} \lra \Phi$ governing the restriction of $g_i$ to $[2, N]^{(k+1)}$. 
Our next result asserts that at some stage of our argument these functions will automatically 
become monotone. 

\begin{lemma}\label{lem:gov}
	If $N\ge 2k$, $m \in [k]$, and $([N]^{(k)},\equiv)$ is governed by $(m+1)$-sets, 
	then the function $g_m^m \colon [2, N]^{(m)} \lra \Phi$ is increasing on $[2,N-k]^{(m)}$.
\end{lemma}

\begin{proof}
Let some integers $2\le a_1 <\dots < a_m < a'_m\le N-k$ satisfying $g_m^m(a_1,\dots,a_m)=1$
be given. We are to prove that $g_m^m(a_1,\dots, a_{m-1}, a'_m)=1$ holds as well.
Set $a_0=1$ and $a_{m+i}=N-k+i$ for every $i\in [k+1-m]$. Our assumption yields
$g_m({a_1},\dots,a_{k+1})=1$, whence
\[ 
	\{a_1,\dots,a_m,a_{m+2},\dots,a_{k+1}\} 
	\equiv 
	\{a_1,\dots,a_{m-1},a_{m+1},\dots,a_{k+1}\}\,. 
\]

This piece of information is among the facts encoded in $f(a_0,\dots,a_{k+1})$
and, as $([N]^{(k)},\equiv)$ is governed by $(m+1)$-sets, the same information is known 
to 
\[
	f(a_0,\dots,a_m,a'_m,a_{m+2},\dots,a_{k+1})\,.
\]
Consequently, we also have
\[  
	\{a_1,\dots,a_m,a_{m+2},\dots,a_{k+1}\} 
	\equiv 
	\{a_1,\dots,a_{m-1},a'_m,a_{m+2},\dots,a_{k+1}\}\,. 
\]
Both displayed equivalences have the same left side, so the transitivity of $\equiv$ leads to
\[
	\{a_1,\dots,a_{m-1},a'_m,a_{m+2},\dots,a_{k+1}\} 
	\equiv 
	\{a_1,\dots,a_{m-1},a_{m+1},\dots,a_{k+1}\}\,,
\]
for which reason 
\[
	g_m^m(a_1, \dots, a_{m-1}, a'_m)
	=
	g_m(a_1, \dots, a_{m-1},a'_m, a_{m+1}, \dots, a_{k+1})
	=
	1\,. \qedhere
\]
\end{proof}

Let us observe that for canonical equivalence relations $\equiv_I$ with $I\subseteq [k]$ the 
auxiliary functions $g_1, \dots, g_k$ are constant. Moreover, if $\eps_1, \dots, \eps_k\in\Phi$
denote the values these functions always attain, then $I=\{i\in [k]\colon \eps_i=-1\}$. 
In the converse direction, if we have an equivalence relation $([N]^{(k)}, \equiv)$ with the property
that such colours $\eps_1, \dots, \eps_k\in\Phi$ exist, then we may 
define $I=\{i\in [k]\colon \eps_i=-1\}$, but it does not follow immediately that~$\equiv$ 
coincides with the canonical equivalence relation~$\equiv_I$. This is why we require an additional 
purging argument. 

\begin{definition}\label{d:p}
	For $i\in [k]$ an equivalence relation $([N]^{(k)},\equiv)$ is said to be {\it $i$-purged} 
	if for all integers $a_1, \dots, a_i, b_1, \dots, b_i, c_{i+1}, \dots, c_{k+1}$
	satisfying
	\begin{enumerate}[label=\rmlabel]
		\item\label{p:1} $1\le a_1 < \dots < a_i$ and $ 1\le b_1 < \dots < b_i$,
		\item\label{p:2} $\max(a_i,b_i) < c_{i+1} < \dots < c_{k+1}\le N$,
		\item\label{p:3} $g_i(a_1,\dots,a_i,c_{i+1},\dots,c_{k+1})
										=g_i(b_1,\dots,b_i,c_{i+1},	\dots,c_{k+1})= -1$,
		\item\label{p:4} and $\{a_1,\dots,a_i,c_{i+1},\dots,c_{k}\}\equiv 
										\{b_1,\dots,b_i,c_{i+1},\dots,c_{k}\}$
	\end{enumerate}
	we have $a_i = b_i$.
\end{definition}

The r\^{o}le of the next lemma is roughly comparable to that of Lemma~\ref{lem:1458}
in the previous section. 
 
\begin{lemma}\label{lem:fin}
	Suppose that the equivalence relation $([N]^{(k)},\equiv)$ is $i$-purged for every $i \in [k]$. 
	If the auxiliary functions $g_1,\dots,g_k$ are constant, then $([N-1]^{(k)},\equiv)$ is canonical.
\end{lemma}

\begin{proof}
	Let $\eps_1, \dots, \eps_k\in \Phi$ be the values which $g_1, \dots, g_k$ always attain
	and set 
	\[
		I = \big\{i\in [k] \colon \eps_i = -1\big\}\,. 
	\]
	It suffices to prove that for any two sets $a= \{ a_1,\dots,a_k\}$ and $b = \{ b_1,\dots,b_k\}$
	from $[N-1]^{(k)}$ whose elements have just been enumerated in increasing order the statement 
	\begin{equation}\label{eq:4940}
		\{a_1,\dots,a_k\} \equiv \{b_1,\dots,b_k\} 
		\,\,\,\Longleftrightarrow\,\,\, 
		\forall i \in I \,\,\, a_i = b_i 
	\end{equation}
	is valid. To this end we argue by induction on the 
	number $\mu = |\{ i \in [k] \colon a_i \neq b_i \}|$ counting the indices where $a$ and $b$ 
	disagree. 
	In the base case $\mu=0$ we have $a=b$ and both sides of the equivalence~\eqref{eq:4940} are true. 
	Now suppose that $\mu\in [k]$ has the property that~\eqref{eq:4940} 
	holds for all pairs of sets from $[N]^{(k)}$ disagreeing less than $\mu$ times. For the 
	induction step we assume that $a, b\in [N]^{(k)}$ disagree with respect to exactly $\mu$ indices. 
	
	Beginning with the backwards implication we assume $a_i = b_i$ for all $i \in I$ 
	and intend to derive $a\equiv b$. 
	Notice that the smallest index $i(\star) \in [k]$ satisfying $a_{i(\star)} \ne b_{i(\star)}$ 
	belongs to~${[k]\setminus I}$, whence $\eps_{i(\star)} =1$. 
	By symmetry we may suppose that $a_{i(\star)} < b_{i(\star)}$. 
	The induction hypothesis implies that the set 
	$c=b\cup\{a_{i(\star)}\}\setminus\{b_{i(\star)}\}$ is equivalent to $a$ with respect 
	to~$\equiv$. Moreover, the definition of  
	\[ 
		g_{i(\star)}(b_1,\dots,b_{i(\star)-1},a_{i(\star)},b_{i(\star)},\dots,b_k)
		=
		\varepsilon_{i(\star)} 
		= 1
	\]
	yields $b\equiv c$, so altogether we have indeed $a\equiv c\equiv b$.

	Now suppose conversely that $a\equiv b$, where $a, b\in [N-1]^{(k)}$. 
	Due to $\mu>0$ there exists a largest 
	integer $i(\star) \in [k]$ such that $a_{i(\star)} \ne b_{i(\star)}$. 
	Applying the assumption that $\equiv$ be $i(\star)$-purged to 
	\begin{enumerate}
		\item[$\bullet$] $a_1<\dots<a_{i(\star)}$, $b_1<\dots<b_{i(\star)}$,
		\item[$\bullet$] and $a_{i(\star)+1}=b_{i(\star)+1}, \dots, a_k=b_k, N$ playing the 
			r\^{o}les of $c_{i(\star)+1}, \dots, c_{k+1}$
	\end{enumerate}
	we infer $\varepsilon_{i(\star)} = +1$, i.e., $i(\star) \notin I$. 
	By symmetry we may again assume $a_{i(\star)} < b_{i(\star)}$. Now 
	\[
		g_{i(\star)}(a_1, \dots, a_{i(\star)}, b_{i(\star)}, a_{i(\star)+1}, \dots, a_k)
		=
		\eps_{i(\star)}
		=
		1
	\]
	implies that the set $c=a\cup\{b_{i(\star)}\}\setminus\{a_{i(\star)}\}$ is equivalent to $a$
	with respect to $\equiv$. 
	Since~$b$ and~$c$ disagree only $\mu-1$ times, the induction hypothesis yields 
	$b_i = c_i$ for all~$i\in I$. Together with $i(\star)\not\in I$ this shows that we have indeed
	$a_i = b_i$ for all~$i \in I$. This completes the induction step and, therefore, the proof 
	of~\eqref{eq:4940}. In particular,~$\equiv$ is canonical. 
\end{proof}

The next result shows that from a quantitative point of view purging is accompanied by a  
polynomial dependence of the involved constants.  

\begin{lemma}\label{lem:pur}
	Let $m\in[k]$, $n\ge k$, and $N\geq (2kn)^{2k}$.
	If an equivalence relation $([N]^{(k)},\equiv)$ is governed by $(m+1)$-sets, 
	then there exists a set $Z\subseteq [N]$ of size $n$ such that 
	$(Z^{(k)},\equiv)$ is $m$-purged.
\end{lemma}

\begin{proof}
The potential counterexamples to $([2, N]^{(k)}, \equiv)$ being $m$-purged can be regarded 
as $(k+m+1)$-tuples of integers satisfying the conditions~\ref{p:1}\,--\,\ref{p:4} 
from Definition~\ref{d:p} and $a_m\ne b_m$. By symmetry it suffices to worry about those 
counterexamples for which $a_m<b_m$. 
Bearing this in mind we define $\Psi$ to be the set of 
all $(k+m)$-tuples 
\[
	(a_1,\dots,a_m,b_1,\dots,b_m,c_{m+1},\dots,c_k)\in [2, N]^{k+m}
\]
such that 
\begin{enumerate}
	\item[$\bullet$] $a_1 < \dots < a_m$, $b_1 < \dots < b_m$, 
	\item[$\bullet$] $a_m < b_m$,
	\item[$\bullet$] $\max(a_m,b_m) < c_{m+1}<\dots < c_k$,
	\item[$\bullet$] $g_m^m(b_1,\dots,b_m)= -1$,
	\item[$\bullet$] and $\{a_1,\dots,a_m,c_{m+1},\dots,c_k\} 
			\equiv \{b_1,\dots,b_m,c_{m+1},\dots,c_k\}$.
\end{enumerate}
The key observation is that two members of $\Psi$ cannot differ in their $(2m)^{\mathrm{th}}$
entry alone. 

\begin{claim}\label{clm:dif}
	If $(a_1,\dots,a_m,b_1,\dots,b_m,c_{m+1},\dots,c_k)$ 
	and $(a_1,\dots,a_m,b_1,\dots,b'_m,c_{m+1},\dots,c_k)$ belong to $\Psi$, then $b_m=b'_m$.
\end{claim}

\begin{proof}
	Assume contrariwise that $b_m < b'_m$. Due to 
	\[ 
		\{b_1,\dots,b'_m,c_{m+1},\dots,c_k\} 
		\equiv 
		\{a_1,\dots,a_m,c_{m+1},\dots,c_k\} 
		\equiv 
		\{b_1,\dots,b_m,c_{m+1},\dots,c_k\}  
	\]
	we know $g_m(b_1,\dots,b_m,b_m^\prime,c_{m+1},\dots,c_k)= 1$, which contradicts 
	$g_m^m(b_1,\dots,b_m)=-1$.
\end{proof}

Now we consider the partition 
\[
	\Psi = \Psi_0 \dcup \dots \dcup \Psi_{m-1}
\]
defined by 
\[
	\Psi_i
	=
	\bigl\{(a_1,\dots,a_m,b_1,\dots,b_m,c_{m+1},\dots,c_k)\in\Psi\colon
		|\{ a_1,\dots,a_m,b_1,\dots,b_{m-1} \}|=m+i\bigr\}
\]
for every $i\in [0, m-1]$. As a consequence of Claim~\ref{clm:dif} 
every $(a_1,\dots,c_k)\in\Psi_i$ can be described uniquely 
by specifying 
\begin{enumerate}[label=\nlabel]
	\item\label{it:41} the $(k+i)$-element subset $X=\{a_1, \dots, a_m, b_1, \dots, b_{m-1}, c_{m+1}, 
			\dots, c_k\}$ of $[2, N]$
	\item\label{it:42} and telling for each of $a_1, \dots, a_m, b_1, \dots, b_{m-1}$
			to which of the $m+i$ members of 
			$Y=\{a_1, \dots, a_m, b_1, \dots, b_{m-1}\}$ they are equal.
\end{enumerate}
There are $\binom{N-1}{k+i}$ possibilities for the first decision. The set $Y$ can be obtained 
from $X$ by removing the $k-m$ largest elements and, in particular, $Y$ is definable from $X$. 
Thus there are at most $(m+i)^{2m-1}$ possibilities for~\ref{it:42}. 
Altogether, theses considerations establish 
\[ 
	|\Psi_i| 
	\le 
	\binom{N-1}{k+i}(m+i)^{2m-1} 
	\le 
	(N-1)^{k+i}(2m)^{2m-1} 
\]
In combination with
\begin{align*}
	\sum_{Z \in [2, N]^{(n)}}|\Psi \cap Z^{k+m}|
	&= 
	\sum_{i =0}^{m-1}\binom{(N-1)-(k+i+1)}{n-(k+i+1)}|\Psi_i| \\
	&\le 
	\binom{N-1}{n}\sum_{i =0}^{m-1}\Big(\frac{n}{N-1}\Big)^{k+i+1}|\Psi_i|
\end{align*}
this implies
\begin{align*}
	\sum_{Z\in [2, N]^{(n)}}|\Psi \cap Z^{k+m}|
	&\le
	\binom{N-1}{n}\sum_{i=0}^{m-1}\frac{n^{k+i+1}}{N-1} (2m)^{2m-1} \\
	&\le 
	\frac{(2mn)^{2k}}{2(N-1)} \binom{N-1}{n}
	<
	\binom{N-1}{n}\,.
\end{align*}
Consequently some set $Z\in [2, N]^{(n)}$ satisfies $\Psi \cap Z^{k+m} = \varnothing$. 
Due to the definition of $\Psi$ every such set $Z$ is $m$-purged.
\end{proof}

\subsection{Plan} \label{subsec:erplan}
Suppose that we have an equivalence relation $([N]^{(k)}, \equiv)$, where $N$ is, roughly speaking, 
$k-1$ times exponential in $n$, and that we want to find a subset $X\subseteq [N]$ of size~$n$ 
such that $(X^{(k)}, \equiv)$ is canonical. The material in the previous subsection suggests the 
following strategy. Letting $f\colon [N]^{(k+2)}\lra\Gamma$ denote the function associated 
with $\equiv$ we want to select appropriate subsets $[N]=Z_{k-1}\supseteq\dots\supseteq Z_0$ such that 
for each $m\in [0, k-1]$ the size of $Z_m$ is, roughly, $m$ times exponential in $n$ and the 
restriction of $f$ to $Z_m^{(k+2)}$ is governed by $(m+3)$-sets. 
Lemma~\ref{lem:gov} tells us that if we construct such a sequence of 
sets~$(Z_m)_{0\le m\le k-1}$ by means of an iterative applications of Lemma~\ref{lem:2150}, then 
we can achieve that the restrictions of $g^2_2, \dots, g^2_k$ to $Z_0^{(2)}$ are monotone. 
Now $|Z_0|$ depends polynomially on $n$, which severely restricts our possibilities as to how 
the argument can be continued. Certainly Lemma~\ref{lem:1731} allows us to synchronise 
$g_2, \dots, g_k$. Moreover, synchronising $g_1$ will only cost an additional square root;
essentially this is due to the formula $\ER^{(1)}(n)\le n^2$. Finally $k$ successive applications of 
Lemma~\ref{lem:pur} allow us to perform a complete purging, and by Lemma~\ref{lem:fin} we will 
thereby be done. 

Now if one completely ignores the last exponent $C_k$ for which one will end up 
proving $\ER^{(k)}(n)\le t_{k-1}(n^{C_k})$, then it is entirely immaterial in which order one 
carries the various steps of the argument out. But from a quantitative point of view it seems 
recommendable to perform each purging step as early as possible. Illustrating this point 
by means of an example we give a brief sketch as to how one can 
prove $\ER^{(4)}(n)\le t_3(Cn^{24})$.
\begin{enumerate}
	\item[$\bullet$] Start with an arbitrary equivalence relation $([N]^{(4)}, \equiv)$, where 
			$N=t_3(Cn^{24})$, and let $f\colon [N]^{(6)}\lra \Gamma$ be its associated 
			function.
	\item[$\bullet$] By Corollary~\ref{cor:0053} there is a set $Z_2\subseteq Z_3=[N]$
	      of size $t_2(C'n^{24})$ on which $f$ is governed by quintuples. 
	      Now $g^4_4$ is monotone by Lemma~\ref{lem:gov}. 
	\item[$\bullet$] Before doing anything else we appeal to Lemma~\ref{lem:pur} and get 
	      a $4$-purged set $P_2\subseteq Z_2$ of size $t_2(C''n^{24})$, where $C''$
	      is extremely close to $C'$. 
	\item[$\bullet$] Next we apply Lemma~\ref{lem:2150} and obtain $Z_1\subseteq P_2$
	      of size $t_1(C'''n^{24})$ such that~$g^3_3$,~$g^3_4$ are monotone and 
	      $f$ is governed by quadruples. 
	 \item[$\bullet$] Now Lemma~\ref{lem:pur} leads us to 
	 		a $3$-purged set $P_1\subseteq Z_1$ 
	 		of size $t_1(C''''n^{24})$.
	\item[$\bullet$] Only now we appeal to Lemma~\ref{lem:2150} again and take $Z_0\subseteq P_1$
			of size $C'''''n^8$ such that~$g^2_2$,~$g^2_3$,~$g^2_4$ are monotone and $f$ is governed 
			by triples. 
\end{enumerate}

The next lemma will tell us how to complete the argument once $Z_0$ is found. 
Before stating it we would like to point out that the two purging steps above only 
influence the constant $C$ we need to start with, but not the final exponent $24$. 
If we would move these two steps to the end of the above list, the argument remained 
valid except for the fact that we needed to start with $24\cdot 8^2$ instead of $24$.  

\begin{lemma}\label{lem:end}
	Let $n\ge k$, $N\ge (k-1)!(2n)^{2k}$, and let $([N]^{(k)},\equiv)$ be an equivalence relation  
	which is governed by triples and $i$-purged for every $i\in [3,k]$.  
	If the functions $g_2^2,\dots,g_k^2$ defined on $[2, N]^{(2)}$ 
	are monotone, then there exists a set $X\subseteq [N]$ 
	of size $n$ such that $(X^{(k)},\equiv)$ is canonical.
\end{lemma}

\begin{proof}
	Due to the monotonicity of $g_2^2,\dots,g_k^2$ and $N-1\ge  (k-1)!(4n^2-1)^{k}$ 
	Lemma \ref{lem:1731} yields a subset $Y\subseteq [2, N]$ of size $|Y|=4n^2-1$ such that 	
	the functions $g_2^2,\dots,g_k^2$ are constant on $Y^{(2)}$. 
	Denote the $k-1$ largest elements of $Y$ by $y_1<\dots<y_{k-1}$  
	and define an equivalence relation $\sim$
	on $Y\setminus\{y_1,\dots,y_{k-1}\}$ by
	\[
		a \sim b 
		\,\,\,\Longleftrightarrow\,\,\, 
		\{a,y_1,\dots,y_{k-1}\} \equiv \{b,y_1,\dots,y_{k-1}\}\,.
	\]
	In view of $n\ge k\ge 2$ we have $|Y\setminus \{y_1,\dots,y_{k-1}\}| = 4n^2 -k \ge (2n-k+1)^2+1$ 
	and, consequently, there exists a set $Z^- \subseteq Y \setminus \{y_1,\dots,y_{k-1}\}$ 
	of size $2n-k+2$ which is either contained in an equivalence class of $\sim$ or consists 
	of numbers that are mutually non-equivalent with respect to $\sim$. In 
	both cases  the set $Z= Z^- \dcup \{y_1,\dots,y_{k-1}\}$ has the property that the restriction 
	of $g^2_1$ to $Z^{(2)}$ is constant. Notice that the cardinality of $Z$ is~$2n+1$. 
	Enumerating the elements of $Z$ in increasing order we write $Z=\{z_1, \dots, z_{2n+1}\}$. 
	
	The remainder of the proof we shall show that  
	\[
		X_\star=\bigl\{z_1, z_3, \dots, z_{2n+1}\bigr\}
	\]
	has the property that $(X_\star^{(k)}, \equiv)$ is $1$-purged and $2$-purged.
	In the light of Lemma~\ref{lem:fin} this will imply that the $n$-element set 
	$X=X_\star\setminus\{z_{2n+1}\}$ is as desired. 
		
	Starting with the former goal we let $\eps_1\in\Phi$ be the constant value that $g^2_1$
	attains on $X_\star^{(2)}$. 
	Because of Definition~\ref{d:p}~\ref{p:3} we may suppose that $\eps_1= -1$. 
	Now let 
	\[
		a_1 \leq b_1 < c_2 < \dots < c_{k+1}
	\]
	be $k+2$ members of $X_\star$ satisfying 
	$\{a_1,c_2,\dots, c_{k}\} \equiv \{b_1,c_2,\dots, c_{k}\}$. 
	As the assumption $a_1< b_1$ would yield $\eps_1=g_1(a_1 ,b_1, c_2, \dots, c_{k})= 1$, 
	we have indeed $a_1=b_1$, as required. 
	
	It remains to show that $(X_\star^{(k)},\equiv)$ is $2$-purged as well. 
	As before, we may suppose that the constant value $\eps_2$ attained by $g_2$ on $X_\star^{(k+1)}$
	is~$-1$. Assume for the sake of contradiction that there exist $k+3$ members $a_1<a_2 <b_2$ 
	and $b_1<b_2<c_3<\dots<c_{k+1}$ of $X$ such that  
	\begin{equation}\label{eq:4809}
		\{a_1,a_2,c_3,\dots,c_{k}\} 
		\equiv 
		\{ b_1,b_2,c_3,\dots,c_k\}\,.
	\end{equation}
	Pick a number $q\in Z$ such that $b_2<q<c_3$. Since $f$ is governed by triples, we have
	\[ 
		f(a_1,a_2,b_1,b_2,c_3,\dots,c_{k}) 
		= 
		f(a_1,a_2,b_1,q,c_3,\dots,c_{k})\,.
	\]
	So~\eqref{eq:4809} entails 
	\[	
		\{a_1,a_2,c_3,\dots,c_{k}\} 
		\equiv  
		\{b_1,q,c_3,\dots,c_k\} 
	\]
	and altogether we obtain 
	\[ 
		\{b_1,b_2,c_3,\dots,c_k\}
		\equiv  
		\{b_1,q,c_3,\dots,c_k\}\,,
	\]
	which leads us to the contradiction
	\[
		\eps_2=g_2(b_1,b_2,q,c_3,\dots,c_k)=1\,. \qedhere 
	\]
\end{proof}

\subsection{Canonisation} \label{subsec:caners}
It remains to execute the plan discussed in the foregoing subsection. 
Given $n\ge k\ge2$ we define positive integers $N_0, P_1, N_1, \dots, P_{k-2}, N_{k-2}, N_{k-1}$ 
according to the following recursive rules. 
\begin{align*}
	N_0&=(k-1)!(2n)^{2k}+(k+1)\,, \\
	P_m&=(2|\Gamma|)^{N_{m-1}^{m+2}} \cdot (2kN_{m-1}^{m})^{kN_{m-1}^{m}} \quad \text{ for }m\in [k-2]\,, \\
	N_{m}&= (2kP_m)^{2k}+(k+1) \quad \text{ for }m\in [k-2]\,,\\
	N_{k-1}&= |\Gamma|^{N_{k-2}^{k+1}} \,.
	\end{align*}
	We want to prove that if $N\ge N_{k-1}$ then for every equivalence relation $([N]^{(k)},\equiv)$ 	
	there exists a set $X\subseteq[N]$ of size $n$ such that $(X^{(k)},\equiv)$ is canonical. 
	This will imply 
	\begin{equation}\label{eq:ende}
		\ER^{(k)}(n)\le N_{k-1}\le t_{k-1}(C_kn^{6k})\,, 
	\end{equation}
	where $C_k$ depends only on $k$. Indeed, we have $P_1=t_1(O(n^{6k}))$ and 
	$N_m=t_m(O(n^{6k}))$ for every $m\in [k-1]$.
	
	Let $f\colon[N]^{k+2}\lra\Gamma$ be the function associated to a given equivalence relation 
	$([N]^{(k)}, \equiv)$, where $N\ge N_{k-1}$. An initial application of Corollary~\ref{cor:0053}
	leads to a set $Z'_{k-2}\subseteq [N]$ such that $({Z'}^{(k)}_{k-2}, \equiv)$ is governed 
	by $(k+1)$-sets and $|Z'_{k-2}|=N_{k-2}$. 
	According to Lemma~\ref{lem:gov} the set $Z_{k-2}\subseteq Z'_{k-2}$
	that arises when one removes the smallest and the $k$ largest elements of~$Z_{k-2}'$ has 
	the property that $g^k_k\colon {Z}^{(k)}_{k-2}\lra \Phi$ is monotone. Notice that 
	$|Z_{k-2}|=N_{k-2}-(k+1)$.

\begin{claim}\label{clm:9999}
	For every $m \in [0,k-2]$ there exists a set $Z_m\subseteq [N]$ of size $|Z_m|=N_m-(k+1)$ 
	with the following properties:
	\begin{enumerate}[label=\rmlabel]
		\item\label{p:5} $(Z_m^{(k)},\equiv)$ is governed by $(m+3)$-sets.
		\item\label{p:6} For every $i \in [m+2,k]$ the function $g_i^{m+2}$ is monotone.
		\item\label{p:7} If $i \in [m+3,k]$, then $(Z_m^{(k)},\equiv)$ is $i$-purged.
	\end{enumerate}
\end{claim}

Observe that the case $m=0$ yields a set $Z_0$ of size $|Z_0|=(k-1)!(2n)^{2k}$.  
Thus Lemma~\ref{lem:end} shows that Claim~\ref{clm:9999} implies~\eqref{eq:ende}.

\begin{proof}[Proof of Claim \ref{clm:9999}]
	We argue by decreasing induction on $m$. In the base case $m=k-2$ clause~\ref{p:7} holds 
	vacuously and the set $Z_{k-2}$ defined above has the desired properties. 
	
	Now suppose that for some $m\in[k-2]$ we have already found a subset $Z_m\subseteq [N]$ 
	of size $N_m-(k+1)=(2kP_m)^{2k}$ fulfilling~\ref{p:5},~\ref{p:6}, and~\ref{p:7}. 
	Owing to Lemma~\ref{lem:pur} there exists a set $Q_m\subseteq Z_m$ 
	of size $|Q_m|=P_m$ such that $(P_{m}^{(k)},\equiv)$ is $(m+2)$-purged.
	Next we apply Lemma~\ref{lem:2150} to $f^{m+3} \colon P_m^{(m+3)} \lra \Gamma$ 
	and the family $(g_i^{m+2})_{m+2\le i\le k}$ of monotone colourings. 
	This yields a set $Z'_{m-1}\subseteq P_m$ of size $N_{m-1}$ such that 
	$({Z'}_m^{(k)},\equiv)$ is governed by $(m+2)$-sets and for every $i \in [m+2,k]$ the 
	function~$g_i^{m+1}$ is monotone. Finally Lemma~\ref{lem:gov} tells us that the 
	set $Z_{m-1}\subseteq Z'_{m-1}$ obtained by removing the smallest and the $k$ largest elements 
	from $Z'_{m-1}$ has the desired properties.
\end{proof}

\subsection{Acknowledgement} We would like to thank the referee for an extremely 
careful reading of this article and providing the alternative proof of 
Lemma~\ref{lem:1730}.

\end{document}